 \documentclass[final,5p,times,twocolumn]{elsarticle}

\usepackage{amsmath,amssymb,amsfonts,amsthm}
\numberwithin{equation}{section}

\usepackage{tikz,tkz-tab} 
\usepackage{moreverb}

\def\Z{{\mathbb Z}}     
\def\R{{\mathbb R}}    
\def\C{{\mathbb C}}    

\graphicspath{{./images/}}

                           \def\cL{{\mathcal L}}

\newcommand{\bs}{\symbol{92}}

\theoremstyle{plain}
\newtheorem{theorem}{Theorem}[section] 
\newtheorem{proposition}[theorem]{Proposition}
\newtheorem{lemma}[theorem]{Lemma}
\newtheorem{corollary}[theorem]{Corollary}

\theoremstyle{definition}

\newtheorem{remark}[theorem]{Remark}

\usepackage{hyperref}

\hypersetup{
    colorlinks=true,
    linkcolor=blue,    
    citecolor=blue,    
    urlcolor=blue      
}

 \journal{}
\date{October 29, 2024}

\begin{document}

\begin{frontmatter}



\title{Prescribed exponential stabilization of scalar neutral differential equations:\\ Application to neural control}

\author[label1,label2]{Cyprien Tamekue}
\author[label1,label2]{Islam Boussaada}
\author[label2]{Karim Trabelsi}
\affiliation[label1]{organization={Université Paris-Saclay, CNRS, CentraleSupélec, Inria, Laboratoire des signaux et systèmes},
            addressline={3 rue Joliot-Curie},
            city={Gif-sur-Yvette},
            postcode={91190},
            country={France}}

\affiliation[label2]{organization={Institut Polytechnique des Sciences Avancées (IPSA)},
            addressline={63 boulevard de Brandebourg},
            city={Ivry-sur-Seine},
            postcode={94200},
            country={France}}



\begin{abstract}

This paper presents a control-oriented delay-based modeling approach for the exponential stabilization of a scalar neutral functional differential equation, which is then applied to the local exponential stabilization of a one-layer neural network of Hopfield type with delayed feedback. The proposed approach utilizes a recently developed partial pole placement method for linear functional differential equations, leveraging the coexistence of real spectral values to explicitly prescribe the exponential decay of the closed-loop solution. While a delayed proportional (P) feedback control may achieve stabilization, it requires higher gains and only allows for a shorter maximum delay compared to the proportional-derivative (PD) feedback control presented in this work. The framework provides a practical illustration of the stabilization strategy, improving upon previous literature results that characterize the solution’s exponential decay for simple real spectral values. This approach enhances neural stability in cases where the inherent dynamics are stable and offers a method to achieve local exponential stabilization with a prescribed decay rate when the inherent dynamics are unstable.

\end{abstract}




\end{frontmatter}




\section{Introduction}

In this paper, we focus on the following scalar differential equation with a single constant delay

\begin{equation}\label{eq:HNN}
       \dot{y}(t) = -(
\nu-\mu) y(t)-k_py(t-\tau)-k_d\dot{y}(t-\tau)
\end{equation}
where $\nu, \mu, k_p, k_d > 0$ and $\tau > 0$ represent system parameters. Equation \eqref{eq:HNN} is of neutral type, and our primary interest lies in analyzing its stability properties. Specifically, we explore the placement of real characteristic values, with an emphasis on conditions that guarantee exponential stability with a prescribed decay rate.

The stability properties of equation \eqref{eq:HNN} have several potential applications, including the prescribed local exponential stabilization of one-layer neural network models of Hopfield type. In particular, these stabilization results can be applied to mitigate the effects of hyperexcitability in such models, which can be associated with seizure-like behavior in individual neurons. Thus, while our focus is on the theoretical stability analysis of equation \eqref{eq:HNN}, our results can directly affect the design of biologically inspired stable inherent neural network models.

Neural networks exhibit complex dynamics crucial for their biological or artificial functionality. The stability of these systems is essential, as instabilities can lead to dysfunctional behaviors such as seizures in biological systems \cite{depannemaecker2021modeling,yaffe2015physiology}. The continuous-time one-layer Hopfield-type equation is a classical model that captures the dynamics of a single neuron or a basic unit of an artificial neural network. The model is described by the following nonlinear differential equation \cite{hopfield1984neurons}
\begin{equation}\label{eq:hopfield model nonlinear intrinsec}
\dot{y}(t) = -\nu y(t) + \mu\tanh(y(t)) + I(t)
\end{equation}
where $\dot{y}(t)$ denotes the rate of change of the neuron's state at time $t$, $\nu$ is a positive parameter reflecting the natural decay rate or leakage of the neuron's membrane potential, $\mu$ is a positive parameter that scales the influence of the activation function, and $I(t)$ is an input.

When $I(t) = 0$, the system dynamics exhibit different stability regimes depending on the relationship between $\nu$ and $\mu$. If $\nu \ge \mu$, the zero equilibrium is stable, with exponential stability when $\nu > \mu$ and asymptotic stability when $\nu = \mu$. Conversely, when $\nu < \mu$, the system becomes hyperexcitable, resulting in instability of the zero equilibrium and the appearance of additional stable equilibria.

To address the hyperexcitability and prevent destabilizing transitions, we can introduce a delayed proportional-derivative (PD) feedback control 
\begin{equation}\label{eq:feedback control}
    I(t) = -k_p y(t-\tau) - k_d \dot{y}(t-\tau)
\end{equation}
where $k_p > 0$ and $k_d > 0$ are the proportional and derivative gains, respectively, and $\tau > 0$ represents a synaptic delay. 
Incorporating this feedback in \eqref{eq:hopfield model nonlinear intrinsec} yields the following delayed differential equation
\begin{equation}\label{eq:hopfield model nonlinear}
    \dot{y}(t) = -\nu y(t) + \mu\tanh(y(t)) - k_p y(t-\tau) - k_d \dot{y}(t-\tau).
\end{equation}
The delayed PD controller introduces time-dependent effects, reflecting physiological synaptic transmission delays \cite{gerstner2014,petkoski2019transmission,sreenivasan2019and}. For further insights into using time delays in modeling biological systems, see, for instance, \cite{beuter1993feedback,gopalsamy2013stability, ruan2006delay}. 
In particular, it is commonly accepted to include time delays in modeling sensory and motoric neural pathways due to the time lag one observes in communication. See, for instance, \cite{stepan2009delay,kuang1993delay} and \cite{wei1999stability} for two delays in neural network modeling. Proportional-integral (PI) controllers have also been proven effective in a firing rate model to control the interspike intervals of a neuron by injecting current through an electrode using a dynamic clamp \cite{miranda2010firing}.

{It is worth noting that a delayed proportional (P) feedback control, given by $ I(t) = -k_p y(t-\tau) $ with $k_p>0$ and $\tau>0$, can potentially stabilize \eqref{eq:hopfield model nonlinear intrinsec} in the hyperexcitable regime where $\mu>\nu$. However, as we demonstrate in Section~\ref{ss:comparison with the P controller}, although the P-controller can achieve the same decay rate as the PD controller in \eqref{eq:feedback control}, it requires higher gains (and, consequently, higher energy costs) and only allows for a shorter maximum allowable delay. This makes it less suitable for stabilizing systems that emulate biological processes, where larger delays are often unavoidable. In contrast, exponential stability is achieved using the delayed PD controller with small values for $k_p$ and $k_d$. The smallness of these gains ensures that the inherent dynamics of the original system are largely preserved, which is crucial for neural networks that aim to mimic biological processes, as it prevents the control strategy from overpowering the system’s natural behaviors and characteristics.}


Equation \eqref{eq:HNN} has been widely discussed in the literature regarding its asymptotic and exponential stability \cite{lien2000stability,hale2013introduction,fridman2001new,
boussaada2022generic,schmoderer2023boundary}. Sufficient delay-independent conditions for stability have been presented. In \cite[Example~1, page~26]{lien2000stability}, the authors proved the global uniform asymptotic stability of equation \eqref{eq:HNN} with constant real coefficients (not dependent upon the delay $\tau$) satisfying $\nu-\mu>0$, $|k_d|<1$, and $|k_p|<(\nu-\mu)\sqrt{1-k_d^2}$ using a Lyapunov functional (LF) and a linear matrix inequality (LMI). It's worth noting that asymptotic analysis of \eqref{eq:HNN} was already considered in \cite[Chapter~9, Section~9.8, page~294]{hale2013introduction} when $k_p=0$, $\nu-\mu>0$, and $|k_d|<1$, see also \cite[Chapter~3, Section~3.3.4, page~69]{fridman2014introduction}. Unfortunately, to the best of the authors' knowledge, no known result using a time-domain approach based on an LF and an LMI can be applied to study the asymptotic stability properties of equation \eqref{eq:HNN} when $\nu-\mu\le 0$.

In this paper, we employ a frequency-domain approach to analyze the distribution of the roots of the characteristic quasipolynomial associated with equation \eqref{eq:HNN}. The characteristic quasipolynomial is given by
\begin{equation}\label{eq:DELTA}
    \Delta_0(s) = s + \nu - \mu + e^{-\tau s}(k_d s + k_p), \qquad s \in \mathbb{C}.
\end{equation}

Our primary objective is to characterize the spectral properties of the quasipolynomial \eqref{eq:DELTA} to ensure that the spectral abscissa is negative, thereby guaranteeing the exponential stability of equation \eqref{eq:HNN}. Recent works have highlighted the importance of \emph{multiplicity-induced dominance (MID)}, where a characteristic root with maximal multiplicity dominates the spectral abscissa \cite{boussaada2020multiplicity,boussaada2022generic,mazanti2021multiplicity}. In particular, it appears that a characteristic root of maximal multiplicity (i.e., equal to the degree of the corresponding quasipolynomial) necessarily defines the spectral abscissa of the system. This property generally occurs in \emph{generic} quasipolynomial and is called GMID. However, in the case of \emph{intermediate} multiplicities,  that is, multiplicities which are less than the quasipolynomial's degree,  the IMID occurs (the largest assigned root corresponds to the spectral abscissa) under some additional conditions, see for instance \cite{benarab2023multiplicity}.
 Since these works, the case of the assignment of a characteristic root with maximal multiplicity was recently addressed and thoroughly characterized in \cite{mazanti2021multiplicity} (generic retarded case) and in \cite{boussaada2022generic} (unifying retarded and neutral cases) for LTI DDEs including a \emph{single delay\/}. 
It is essential to point out that the multiplicity of a given spectral value itself is not essential. Still, its connection with the dominance of this root is a meaningful tool for control
synthesis. Namely,  it is shown that, under appropriate conditions, the coexistence of exactly the maximal number of distinct negative zeros of quasipolynomial of reduced degree guarantees the exponential stability of the zero solution of the corresponding time-delay system, a property called \emph{Coexisting Real Roots Induced Dominance (CRRID)}, see for instance \cite{amrane2018qualitative,bedouhene2020real, schmoderer2023boundary,schmoderer2023insights}. These properties opened an interesting perspective in control through the so-called \emph{partial pole placement\/} method, that is, imposing the multiplicity or the coexistence of simple real characteristic root of the closed-loop system by an appropriate choice of the controller gains guarantees the exponential stability of the closed-loop system with a prescribed decay rate. 


The stability properties of \eqref{eq:HNN} have also been linked to the problem of boundary PI control for the transport equation, as explored in \cite{schmoderer2023boundary}. In that work, the authors demonstrated the effectiveness of the CRRID properties in achieving exponential stabilization of the closed-loop transport equation with a prescribed decay rate. In contrast, the results in \cite{Coron2015feedback}, which rely on Lyapunov functions, do not offer the same guarantee regarding the decay rate.

The contributions of this paper are as follows: First, it refines recent results on the \emph{Coexisting Real Roots Induced Dominance (CRRID)} property for the first-order neutral functional differential equations, providing simplified proofs and insights into the conditions required for CRRID to hold. Specifically, we provide a more straightforward proof for the Generic CRRID (GCRRID) property, covering arbitrary distributions of real roots, and derive necessary and sufficient conditions for the Intermediate CRRID (ICRRID) property, including cases where unintentional additional real roots coexist. Moreover, we generalize and simplify the proof of the Intermediate MID (IMID) property from \cite{benarab2023multiplicity}. Secondly, we extend the results from \cite{schmoderer2023boundary}, providing a detailed characterization of the spectrum distribution when the quasipolynomial has three real spectral values or exactly two real ones. Furthermore, in the single-delay case,  when the  Frasson-Verduyn Lunel's \cite{frasson2003large} sufficient condition for the dominance of a simple real spectral value is not met, we prove that the CRRID property remains valid. Finally, by leveraging the CRRID property, we characterize the spectral abscissa's negativity regarding system parameters and establish an explicit and improved exponential estimate for the closed-loop system's solution compared to previous results.

The remainder of this paper is organized as follows: Section \ref{s:problem settings} provides prerequisites in complex analysis and formulates the problem settings. Section \ref{s:main results} presents the main results and their corresponding proofs. Section \ref{s:application to neural network modelling} translates the main results into the modeling and stabilization of one-layer neural networks.


\section{Problem settings and prerequisites}\label{s:problem settings}

Throughout the following, our focus is on studying the asymptotic behavior of solutions of the general scalar neutral functional differential equation (NFDE) in Hale’s form 
\begin{equation}\tag{NDE}\label{eq:NDE}
        \frac{d}{dt}(y(t)+\alpha y(t-\tau)) = -ay(t)-\beta y(t-\tau)
\end{equation}
with initial function $y_0\in C^0([-\tau,0])$.
To this equation corresponds the characteristic quasipolynomial function given by
\begin{equation}\label{eq:Delta}
    \Delta(s) = s+a+e^{-\tau s}(\alpha s+\beta) \quad (s\in\C)
\end{equation}
where $(a,\alpha,\beta)\in\R^3$ and $\tau>0$. It is known that the degree of the quasipolynomial $\Delta$---the sum of the degrees of the involved polynomials plus the number of delays---is equal to three. Moreover, thanks to the P\'{o}lya-Szeg\"{o} bound \cite[Problem~206.2, page~144]{polya1972analysis}, the degree of $\Delta$ is a sharp bound for the number of real roots counting multiplicities of the quasipolynomial $\Delta$.

In the next section, we study the spectral properties of the quasipolynomial function $\Delta$, focusing on characterizing its rightmost root. Consider a complex value $s_0\in\C$ such that $\Delta(s_0)=0$. We say that $s_0$ is  a \textit{dominant} (respectively, \textit{strictly dominant}) root of $\Delta$ if the following holds:
\begin{equation}\label{eq:dominacy}
    \forall s\in\C\bs\{s_0\},\quad\Delta(s) = 0\implies
    \begin{cases}
        \Re(s)\le\Re(s_0),\\
        \mbox{respectively}\\
        \Re(s)<\Re(s_0).
    \end{cases}
\end{equation}
Despite the challenging question of characterizing the spectral abscissa of \eqref{eq:NDE},   Frasson-Verduyn Lunel in \cite[Lemma~B1]{frasson2003large} provides a test for determining the simplicity and dominance of real spectral values in the multi-delay scalar neutral equations, offering fundamental insights. Notice that such a characterization is closely related to the exponential/asymptotic behavior of the equation \eqref{eq:NDE}.

According to Frasson-Verduyn Lunel’s lemma when restricted to the single-delay case, a real root $s_0$ of $\Delta$ is simple and dominant if  $V(s_0)<1$ where $V$  is a specific functional construct derived from $\Delta$. In other words, the condition is as follows; see, \cite[Lemma~5.1]{frasson2003large}.
\begin{lemma}\label{lem:Frasson-Verduyn Lunel}
    Suppose that there exists a real zero $s_0$ of $\Delta$. If $V(s_0)<1$, then $s_0$ is a real simple dominant zero of $\Delta$. Here,
    \[
    V(s_0) = (|\alpha|(1+|s_0|\tau)+|\beta|\tau)e^{-s_0\tau}
    \]
    where $(\alpha,\beta)\in\R^2$ and $\tau>0$.
\end{lemma}


Before starting the study of the localisation of the quasipolynomial $\Delta$ roots, let us state the following lemma the proof of which can be found in \cite[Problem~77, page~46]{polya2012problems}.

\begin{lemma}[Descartes' rule of signs]\label{lem:Descartes' rule of signs}
    Let $a_1$, $a_2$, $a_3$, $\lambda_1$, $\lambda_2$, $\lambda_3$ be real constants, such that $\lambda_1<\lambda_2<\lambda_3$. Denote by $Z$ the number of real zeros of the entire function
    \begin{equation}
        \mathcal{F}(x) = a_1e^{\lambda_1 x}+a_2e^{\lambda_2 x}+a_3e^{\lambda_3 x}
    \end{equation}
    and by $C$, the number of sign changes in the sequence of numbers $a_1$, $a_2$, $a_3$. Then, $C-Z$ is a non-negative even integer.
\end{lemma} 
One also has the following important result, which extends \cite[Corollary~2]{schmoderer2023boundary}. It includes the case where we exactly assign two real spectral values to $\Delta$.
\begin{lemma}\label{lem:PS bound}
    Let $\Delta$ be the quasipolynomial defined by \eqref{eq:Delta} with real coefficients. Let $\eta\in\{0, 1\}$, if $\Delta$ admits exactly $3-\eta$ real roots, then any root $x+i\omega\in\C$ of $\Delta$ with $\omega\neq 0$ satisfies 
    \begin{equation}
        |\omega|\ge{(2-\eta)\pi}/{\tau}.
    \end{equation}
\end{lemma}
\begin{proof}
    Recall from \cite[Problem~206.2, page~144]{polya1972analysis} that if $M_{\alpha,\beta}$ denotes the number of roots of $\Delta$ contained in the horizontal strip $\{s\in\C\mid\alpha\le\Im(s)\le\beta\}$ counting multiplicities, then the following bound holds
    \begin{equation}\label{eq:PS bound}
        \frac{\tau(\beta-\alpha)}{2\pi}-3\le M_{\alpha,\beta}\le \frac{\tau(\beta-\alpha)}{2\pi}+3.
    \end{equation}
    To complete the proof of the lemma, we argue by contradiction. Assume that $\Delta$ admits $3-\eta$ real roots, and let $s_0:=x+i\omega\in\C$ be a root of $\Delta$ with $\omega\neq 0$ and $|\omega|<(2-\eta)\pi/\tau$. Then there exists $\varepsilon>0$ such that $|\omega|<{(2-\eta)\pi}/{\tau}-\varepsilon$. Since $\Delta$ has real coefficients, the complex conjugate of $s_0$ is also a root of $\Delta$, both belonging to the horizontal strip $\{s\in\C\mid-\frac{(2-\eta)\pi}{\tau}+\varepsilon\le\Im(s)\le\frac{(2-\eta)\pi}{\tau}-\varepsilon\}$. It follows that $\Delta$ admits at least  $5-\eta$ roots in this strip, which is inconsistent, since by the P\'{o}lya-Szeg\"{o} bound \eqref{eq:PS bound}, the number of zero in this strip satisfies 
    \begin{equation*}
        M_{-\frac{(2-\eta)\pi}{\tau}+\varepsilon,\frac{(2-\eta)\pi}{\tau}-\varepsilon}\le 5-\eta-\frac{\varepsilon \tau}{\pi}< 5-\eta.\qedhere
    \end{equation*}
\end{proof}

Let us investigate the coexistence of---non-necessarily equidistributed---three real roots for the quasipolynomial $\Delta$. Due to the linearity of $\Delta$ with respect to its coefficients $a$, $\alpha$ and $\beta$, one reduces the system $\Delta(s_1)=\Delta(s_2)=\Delta(s_3) = 0$ to the linear system 
\begin{equation}\label{eq:cramer system  degree 3}
A_{\tau,3}(s_1,s_2,s_3)X = B,    
\end{equation}
where $B = -(s_1e^{\tau s_1},s_2e^{\tau s_2}, s_3e^{\tau s_3})^t$, $X=(\alpha,\beta,a)^{t}$ and 
\begin{equation}
A_{\tau,3}(s_1,s_2,s_3)= \begin{bmatrix}
s_1 & 1 & e^{\tau s_1} \\
s_2 & 1 & e^{\tau s_2} \\
s_3 & 1 & e^{\tau s_3}
\end{bmatrix}.
\end{equation}

Using \cite[Theorem~2]{bedouhene2020real}, one immediately obtains that the determinant of the structured functional Vandermonde-type matrix  $A_{\tau,3}(s_1,s_2,s_3)$ is given by
\begin{equation}\label{eq:determinant degree 3}
    D_{\tau,3}(s_1,s_2,s_3)=\tau^2(s_1-s_2)(s_1-s_3)(s_2-s_3)F_{-\tau,2}
\end{equation}
where $F_{-\tau,2}:=F_{-\tau,2}(s_1,s_2,s_3)$ is defined by
\begin{equation}\label{eq:recursive function}
    F_{-\tau,2}= \int_{0}^{1}\int_{0}^{1}(1-t_1)e^{\tau(t_1s_1+(1-t_1)(t_2s_2+(1-t_2)s_3))}dt_1dt_2>0.
\end{equation}
By integrating \eqref{eq:recursive function}, one may carefully check that
\begin{equation}\label{eq:recursive function 1}
    F_{-\tau,2} = \frac{e^{\tau s_3}(s_1-s_2)+e^{\tau s_2}(s_3-s_1)+e^{\tau s_1}(s_2-s_3)}{\tau^2(s_1-s_3)(s_2-s_3)(s_1-s_2)}.
\end{equation}

For distinct real spectral values $s_3\neq s_2\neq s_1$, one deduces from \eqref{eq:determinant degree 3} and \eqref{eq:recursive function} that for every $\tau>0$, $D_{\tau,3}(s_1,s_2,s_3)\neq 0$. It follows that \eqref{eq:cramer system  degree 3} is a Cramer system, and one can immediately compute the coefficients $\alpha$, $\beta$, and $a$ using the Cramer formula. More precisely, one has the following.
\begin{lemma}\label{lem:coeffs}
    For a fixed $\tau>0$, the quasipolynomial $\Delta$ admits three distinct real spectral values $s_3$, $s_2$ and $s_1$ if, and only if, the coefficients $\beta$, $\alpha$ and $a$ are respectively given by
        \begin{equation}\label{eq:beta}
    \beta(\tau) = \frac{1}{D_{\tau,3}(s_1,s_2,s_3)}\det\begin{bmatrix}
s_1 & -s_1 e^{\tau s_1} & e^{\tau s_1} \\
s_2 & -s_2 e^{\tau s_2} & e^{\tau s_2} \\
s_3 & -s_3 e^{\tau s_3} & e^{\tau s_3} \\
\end{bmatrix}.
\end{equation}
\begin{equation}\label{eq:alpha}
    \begin{aligned}
        \alpha(\tau) &= \frac{1}{D_{\tau,3}(s_1,s_2,s_3)}\det\begin{bmatrix}
            -s_1e^{\tau s_1} & 1 & e^{\tau s_1} \\
            -s_2e^{\tau s_2} & 1 & e^{\tau s_2} \\
            -s_3e^{\tau s_3} & 1 & e^{\tau s_3}
        \end{bmatrix}\\
        &= \frac{F_{-\tau,2}(s_1+s_2,s_1+s_3,s_2+s_3)}{F_{-\tau,2}(s_1,s_2,s_3)}.
    \end{aligned}
\end{equation}
\begin{equation}\label{eq:a1}
    \begin{aligned}
        a(\tau) &= \frac{1}{D_{\tau,3}(s_1,s_2,s_3)}\det\begin{bmatrix}
s_1 & 1 & -s_1 e^{\tau s_1} \\
s_2 & 1 & -s_2 e^{\tau s_2} \\
s_3 & 1 & -s_3 e^{\tau s_3} \\
\end{bmatrix}\\
        &= -s_1-\frac{F_{-\tau,1}(s_2,s_3)}{\tau F_{-\tau,2}(s_1,s_2,s_3)}\\
        &=-s_2-\frac{F_{-\tau,1}(s_1,s_3)}{\tau F_{-\tau,2}(s_1,s_2,s_3)}\\
        &=-s_3-\frac{F_{-\tau,1}(s_1,s_2)}{\tau F_{-\tau,2}(s_1,s_2,s_3)}.
    \end{aligned}
\end{equation}
Here, one has for all $(u, v)\in\R^2$,
\begin{equation}\label{eq:F 1}
   F_{-\tau,1}(u,v) := \int_0^1e^{\tau(tu+(1-t)v)}dt=\frac{e^{\tau u}-e^{\tau v}}{\tau(u-v)}>0.
\end{equation}
\end{lemma}

\begin{proof}
    The Relation \eqref{eq:beta} and the first identities in relations \eqref{eq:alpha}-\eqref{eq:a1} follow directly by applying the Cramer formulas to the Cramer system \eqref{eq:cramer system  degree 3}. To obtain the others in \eqref{eq:alpha}-\eqref{eq:a1}, one may use \eqref{eq:determinant degree 3}, \eqref{eq:recursive function 1}, \eqref{eq:F 1} and the following.
    \begin{equation*}
    \begin{aligned}
        \Lambda_1 &= e^{\tau(s_1+s_2)}(s_1-s_2)+e^{\tau(s_1+s_3)}(s_3-s_1)+e^{\tau(s_2+s_3)}(s_2-s_3)\\
        &= \tau^2(s_1-s_2)(s_1-s_3)(s_2-s_3)F_{-\tau,2}(s_1+s_2,s_1+s_3,s_2+s_3)
    \end{aligned}
\end{equation*}
     \begin{equation*}
    \begin{aligned}
        \Lambda_2 &= e^{\tau s_3}(s_2-s_1)s_3+e^{\tau s_2}(s_1-s_3)s_2+e^{\tau s_1}(s_3-s_2)s_1\\
        &= -s_1D_{\tau,3}(s_1,s_2,s_3)-(e^{\tau s_2}-e^{\tau s_3})(s_1-s_2)(s_1-s_3)\\
        &=-s_2D_{\tau,3}(s_1,s_2,s_3)-(e^{\tau s_1}-e^{\tau s_3})(s_1-s_2)(s_2-s_3)\\
        &=-s_3D_{\tau,3}(s_1,s_2,s_3)-(e^{\tau s_1}-e^{\tau s_2})(s_1-s_3)(s_2-s_3)
    \end{aligned}
\end{equation*}
where
  \begin{equation*}
       \Lambda_1:=\det\begin{bmatrix}
-s_1e^{\tau s_1} & 1 & e^{\tau s_1} \\
-s_2e^{\tau s_2} & 1 & e^{\tau s_2} \\
-s_3e^{\tau s_3} & 1 & e^{\tau s_3} \\
\end{bmatrix},\quad    \Lambda_2:=  \det\begin{bmatrix}
s_1 & 1 & -s_1 e^{\tau s_1} \\
s_2 & 1 & -s_2 e^{\tau s_2} \\
s_3 & 1 & -s_3 e^{\tau s_3} \\
\end{bmatrix}.
 \end{equation*}
 This completes the proof of the lemma.
\end{proof}
The next lemma is a key ingredient in simplifying the proofs of our main results in Section~\ref{ss:3 DRSV}.
\begin{lemma}\label{lem:properties on alpha and beta GCRRID}
   Assume that the quasipolynomial $\Delta$ has three real spectral values $s_3<s_2<s_1$. Then, the following holds
        \begin{equation}\label{eq:asymptotic alpha 3RR}
      \begin{split}
 	 0<\alpha<e^{\tau s_1}&.\\
 \alpha<e^{\tau x}&\qquad(\forall x\ge s_1).
  \end{split}
 \end{equation}
\end{lemma}
\begin{proof}
    Let $x\ge s_1$. Then, one has $e^{\tau x}-\alpha>0$ since
 \begin{equation*}
     e^{\tau x}-\alpha = \frac{e^{\tau x}F_{-\tau,2}(s_2,s_3,s_1)-F_{-\tau,2}(s_1+s_2,s_2+s_3, s_1+s_3)}{F_{-\tau,2}(s_2,s_3,s_1)}
 \end{equation*}
 and $F_{-\tau,2}(s_1,s_2,s_3)=F_{-\tau,2}(s_2,s_3,s_1)$ by \cite[Lemma~4]{bedouhene2020real} and the fact that for every $(t_1, t_2)\in[0, 1]^2$, one has
 \begin{equation*}
     \begin{split}
         x+t_1s_2+(1-t_1)(t_2s_3+(1-t_2)s_1))-t_1(s_1+s_2)&-\\
         (1-t_1)(t_2(s_2+s_3)-(1-t_2)(s_1+s_3)))&=\\
         x-t_1s_1-t_2(1-t_1)s_2-(1-t_1)(1-t_2)s_3 & >\\
         (1-t_1)s_1-t_2(1-t_1)s_1-(1-t_1)(1-t_2)s_1& =0.\qedhere
     \end{split}
 \end{equation*}
\end{proof}

\begin{remark}
    From Lemmas~\ref{lem:coeffs} and \ref{lem:properties on alpha and beta GCRRID}, if $s_1<0$ then for every $\tau>0$, one has $\beta(\tau)=-\alpha s_1+\zeta e^{\tau s_1}>0$ and $a(\tau)$ may change sign. Here,
    \begin{equation}\label{eq:function zeta}
        \zeta:=\zeta(\tau)=\frac{F_{-\tau,1}(s_2,s_3)}{\tau F_{-\tau,2}(s_1,s_2,s_3)}>0.
    \end{equation}
\end{remark}

\section{Main results}\label{s:main results}

In this section, we establish our main results. Section~\ref{ss:3 DRSV} discusses the GCRRID property for the quasipolynomial $\Delta$ and derives a new and simpler proof for its validity. In Section~\ref{ss:2 DRSV}, the ICRRID property is investigated. Next, Section~\ref{ss:Frasson's sufficient conditions} provides an example of a simple dominant spectral value violating the sufficient condition established by Frasson-Verduyn Lunel, thereby emphasizing the non-necessary nature of the latter condition.  Additionally, we will present a simpler proof of the exponential estimates for solutions of equation \eqref{eq:NDE} in Section~\ref{ss:exponential estimates}.

\subsection{Assigning three distinct real spectral values}\label{ss:3 DRSV}

This section discusses the GCRRID property for the quasipolynomial $\Delta$. It involves assigning the maximal number of distinct real spectral values $s_3 < s_2 < s_1$ to $\Delta$ and proving that $s_1$ is a dominant root. We shall then use this result to establish necessary and sufficient conditions to ensure that $s_1$ is negative, which is essential to the exponential stability of equation \eqref{eq:NDE}. Finally, we explicitly characterize the spectrum of $\Delta$ when it admits three real roots $s_3<s_2<s_1$.

The first main result of this section provides a comprehensive proof of \cite[Theorem~5]{schmoderer2023boundary}.

\begin{theorem}[Dominance of a real root]\label{thm:dominance of a real root 3 rr}
    Assume that $\Delta$ admits three real spectral values $s_3<s_2<s_1$. Then, the spectral value $s_1$ is a strictly dominant root of $\Delta$.
\end{theorem}
\begin{proof}
Fix $\tau>0$. It follows from Lemma~\ref{lem:coeffs} that $a+s_1<0$ and $\alpha>0$. We argue by contradiction. Assume that there exists $s_0:=x+i\omega\in\C$ such that $\Delta(s_0) = 0$ and $x\ge s_1$. In particular, $\omega\neq 0$, since we have already assigned three real spectral values $s_3<s_2<s_1$ to $\Delta$. From $\Delta(s_1) = 0$, one gets $\beta = -\alpha s_1-e^{\tau s_1}(a+s_1)$. Therefore, $ \Delta(s_0) = 0$ if, and only if,
\begin{equation}\label{eq:delta s 0 vanishes  brouillon}
    e^{\tau s_0}s_0+e^{\tau s_0}a+\alpha s_0 = \alpha s_1+e^{\tau s_1}(a+s_1).
\end{equation}
By taking the real and imaginary parts in \eqref{eq:delta s 0 vanishes  brouillon}, one gets
 \begin{equation}\label{eq:s 0 and its cc brouillon}
      \begin{split}
 	(a+x)\cos(\tau \omega)-\omega\sin(\tau\omega)&=e^{-\tau (x-s_1)}(a+s_1)-\alpha(x-s_1)e^{-\tau x}\\
(a+x)\sin(\tau \omega)+\omega\cos(\tau\omega)&=-\alpha\omega e^{-\tau x}.
  \end{split}
 \end{equation}
 By squaring each equality in \eqref{eq:s 0 and its cc  brouillon} and adding them yields
 \begin{equation}\label{eq:omega positive real dominance}
     \omega^2=\frac{(e^{\tau s_1}(a+s_1)-\alpha(x-s_1))^2-(a+x)^2e^{2\tau x}}{e^{2\tau x}-\alpha^2}
 \end{equation}
  which is well defined for every $x\ge s_1$ by Lemma~\ref{lem:properties on alpha and beta GCRRID}.
Let us prove that $\omega^2$ given by \eqref{eq:omega positive real dominance} satisfies $\omega^2<1/\tau^2$. To do so, define for all $x\ge s_1$ the function
 \begin{equation}
     \chi(x) = (e^{\tau s_1}(a+s_1)-\alpha(x-s_1))^2-(a+x)^2e^{2\tau x}-\frac{(e^{2\tau x}-\alpha^2)}{\tau^2}.
 \end{equation}
We want to prove that $\chi(x)<0$ for every $x\ge s_1$. On the one hand, one has
\begin{eqnarray}\label{eq:chi 0 and limit at infinity}
     \chi(s_1)=-(e^{2\tau s_1}-\alpha^2)\tau^{-2}<0,\quad\lim_{x\to\infty}\chi(x)  =-\infty
 \end{eqnarray}
 owing to Lemma~\ref{lem:properties on alpha and beta GCRRID}. On the other hand, function $\chi$ is infinitely derivable on $[s_1,\infty)$, so that for all $x\ge s_1$,
\begin{equation}\label{eq:second derivative of chi}
\hspace{-0.3cm}\chi''(x)= -2e^{2\tau x}\left(2\left(\tau(a+x)+1\right)^2+e^{-2\tau x}(e^{2\tau x}-\alpha^2)\right)<0
\end{equation}
since the two terms between the big brackets of $\chi''(x)$ are positive by Lemma~\ref{lem:properties on alpha and beta GCRRID}. One also has for all $x\ge s_1$,
 \begin{equation}\label{eq:first derivative of chi}
     \begin{aligned}
        \chi'(x) &= -2\alpha((a+s_1)e^{\tau s_1}-\alpha(x-s_1))-2\tau (a+x)^2e^{2\tau x}\\
        & -2(a+x)e^{2\tau x}-2\tau^{-1}e^{2\tau x}
    \end{aligned}
\end{equation}
\begin{equation*}
    \begin{cases}
          \begin{aligned}
        \chi'(s_1)&= -2\tau\left((a+s_1)e^{\tau s_1}+\frac{\alpha+e^{\tau s_1}}{2\tau}\right)^2 -\frac{(e^{\tau s_1}-\alpha)(\alpha+3 e^{\tau s_1})}{2\tau}
    \end{aligned}\\
    \lim\limits_{x\to\infty}\chi'(x)  =-\infty.
    \end{cases}
\end{equation*}
Here, $\chi'(s_1)<0$, as the sum of two negative terms by Lemma~\ref{lem:properties on alpha and beta GCRRID}. It follows from \eqref{eq:second derivative of chi} and \eqref{eq:first derivative of chi} that
\begin{equation}\label{eq:sign first derivative of chi}
    \chi'(x)<0 \quad(\forall x\ge s_1).
\end{equation}
Finally, by combining \eqref{eq:sign first derivative of chi} and \eqref{eq:chi 0 and limit at infinity}, one gets that 
\begin{equation}\label{eq:sign of chi}
    \chi(x)<0 \quad(\forall x\ge s_1).
\end{equation}
The latter is equivalent to $\omega^2<\tau^{-2}$ with $\omega^2$ defined in \eqref{eq:omega positive real dominance}. So, we have shown that if $x+i\omega$ is a root of $\Delta$ with $x\ge s_1$ and $\omega\neq 0$, then $|\omega|<\tau^{-1}$. The latter is inconsistent since one has necessarily $|\omega|\ge 2\pi/\tau$ owing to Lemma~\ref{lem:PS bound}. 
\end{proof}
\begin{remark}\label{rmk:observation on alpha 3RR}
    It is worth noting that the key property on $\alpha$ required in the proof of Theorem~\ref{thm:dominance of a real root 3 rr} is that $e^{2\tau x}-\alpha^2\ge 0$ for all $x\ge s_1$ which is satisfied owing to Lemma~\ref{lem:properties on alpha and beta GCRRID}.
\end{remark}
The second main result provides the necessary and sufficient conditions on the delay $\tau$ and the coefficient $a$ to guarantee that a dominant root $s_1$ is negative when three real spectral values $s_3<s_2<s_1$ are formally assigned to $\Delta$.

\begin{theorem}[Negativity of a dominant root]\label{thm:negativity generic case}
     Assume that $\Delta$ admits three real spectral values $s_3<s_2<s_1$. Then, $s_1<0$ if, and only if, there is a unique $\tau_{*}>0$ such that $a(\tau_{*})=0$, i.e.,
    \begin{equation}
        a(\tau)\begin{cases}
            =0\quad\mbox{if}\quad\tau=\tau_{*}\cr
            <0\quad\mbox{if}\quad\tau<\tau_{*}\cr
            >0\quad\mbox{if}\quad\tau>\tau_{*}.
        \end{cases}
    \end{equation}
    More precisely, it holds 
    \begin{equation}
       \hspace{-0.25cm} s_1\begin{cases}
            =-\zeta(\tau)\quad\mbox{if}\quad\tau=\tau_{*}\cr
            >-\zeta(\tau)\quad\mbox{if}\quad\tau<\tau_{*}\cr
            <-\zeta(\tau)\quad\mbox{if}\quad\tau>\tau_{*},
        \end{cases}\quad 
        \zeta(\tau):=\frac{F_{-\tau,1}(s_2,s_3)}{\tau F_{-\tau,2}(s_1,s_2,s_3)}.
    \end{equation}

\end{theorem}

\begin{proof}
     Let us assume that $s_1<0$. Since $\tau\in(0,\;\infty)\mapsto a(\tau)$ is continuous with respect to $\tau$, one may use the intermediate value theorem. Combining \eqref{eq:a1} and \eqref{eq:recursive function 1}, one gets
   \begin{equation*}
        a(\tau) = \frac{e^{-\tau(s_1- s_3)}(s_1-s_2)s_3+e^{-\tau(s_1- s_2)}(s_3-s_1)s_2+(s_2-s_3)s_1}{e^{-\tau(s_1- s_3)}(s_2-s_1)+e^{-\tau(s_1- s_2)}(s_1-s_3)+(s_3-s_2)}.
    \end{equation*}
Consequently, $a(\tau)\to -s_1>0$ when $\tau\to\infty$. Next, \eqref{eq:recursive function}, \eqref{eq:F 1}, and the continuity of $F_{-\tau, 1}$ and $F_{-\tau, 2}$ with respect to $\tau$ imply that $a(\tau)\to-\infty$ as $\tau\to 0$. It follows that there exists at least one $\tau_{*}>0$ such that $a(\tau_{*})=0$. To show that $\tau_{*}$ is unique, one applies Lemma~\ref{lem:Descartes' rule of signs}. We observe that $a(\tau)=0$ if, and only if, its numerator vanishes at $\tau$. Now, let
  \begin{equation}
      F(\tau) = e^{\tau s_3}(s_1-s_2)s_3+e^{\tau s_2}(s_3-s_1)s_2+e^{\tau s_1}(s_2-s_3)s_1
  \end{equation}
  be the numerator of $e^{\tau s_1}a(\tau)$. Then, $F$ is analytic in $\tau$, and one has from $s_3<s_2<s_1$ that $(s_1-s_2)s_3<0$, $(s_3-s_1)s_2>0$ and $(s_2-s_3)s_1<0$. Let $C$ denote the number of sign changes in the sequence of real numbers $(s_1-s_2)s_3$, $(s_3-s_1)s_2$, and $(s_2-s_3)s_1$, then $C=2$. Similarly, if $Z$ represents the number of real zeros of the entire function $F$, then $C-Z=2-Z\ge 0$ and therefore $Z=2$ by Lemma~\ref{lem:Descartes' rule of signs}. Since $F(0)=0$ and $F(\tau_{*})=0$, the uniqueness of $\tau_{*}>0$ follows.
  
  Conversely, assume the existence of a unique $\tau_{*}>0$ such that $a(\tau_{*})=0$. One immediately deduces from \eqref{eq:recursive function}, \eqref{eq:a1} and \eqref{eq:F 1} that
  \begin{equation*}
      s_1 = -\frac{F_{-\tau_{*},1}(s_2,s_3)}{\tau_{*} F_{-\tau_{*},2}(s_1,s_2,s_3)}<0.\qedhere
  \end{equation*}
\end{proof}
\begin{remark}
    Theorem~\ref{thm:negativity generic case} states that the exponential stability of equation \eqref{eq:NDE} may be achieved with a prescribed decay rate, even though $a\le 0$. Time-domain techniques based on a Lyapunov functional and LMI do not cover this case. See, for instance, \cite{lien2000stability,hale2013introduction,fridman2001new}.
\end{remark}
In the case of equidistributed real spectral values $s_3<s_2<s_1$, the delay $\tau_{*}>0$ that enables the design of $s_1<0$ may be explicitly computed, as is stated hereafter.
\begin{corollary}\label{cor:equidistributed three real roots}
    Assume that the quasipolynomial $\Delta$ admits three equidistributed real spectral values $s_k = s_1-(k-1)d$, for $d>0$. Then, $s_1<0$ if, and only if, the delay $\tau_{*}>0$ in Theorem~\ref{thm:dominance of a real root 3 rr} reads
    \begin{equation}
        \tau_{*} = \frac{1}{d}\ln\left(\frac{s_3}{s_1}\right)>0.
    \end{equation}
\end{corollary}
\begin{proof}
		By Theorem~\ref{thm:negativity generic case}, $s_1<0$, if and only if, there exists one, and only one, $\tau_{*}>0$ such that $a(\tau_{*})=0$, i.e.,
		\begin{equation}\label{eq:23}
			s_1 = -\frac{F_{-\tau_{*},1}(s_2,s_3)}{\tau_{*} F_{-\tau_{*},2}(s_1,s_2,s_3)}.
		\end{equation}
		Letting $s_2=s_1-d$ and $s_3=s_1-2d$ with $d>0$, one finds
		\begin{equation}\label{eq:24}
			F_{-\tau_{*},2}(s_1,s_2,s_3) = \frac{e^{\tau_{*}s_1}(1-e^{-\tau_{*}d})^2}{2\tau_{*}^2d^2}
		\end{equation}
        \begin{equation}\label{eq:25}
            F_{-\tau_{*},1}(s_2,s_3) = \frac{e^{\tau_{*}s_1}e^{-\tau_{*}d}(1-e^{-\tau_{*}d})}{\tau_{*}d}.
        \end{equation}
		By injecting \eqref{eq:24}-\eqref{eq:25} into \eqref{eq:23}, one gets
		\begin{equation}
			s_1 = \frac{-2d}{e^{\tau_{*}d}-1}\quad\Longleftrightarrow\quad \tau_{*} = \frac{1}{d}\ln\left(\frac{s_3}{s_1}\right)>0,
		\end{equation}
		since $s_3<s_1<0$.
	\end{proof}
\begin{figure}[t]
		\centering
		\includegraphics[width=\linewidth]{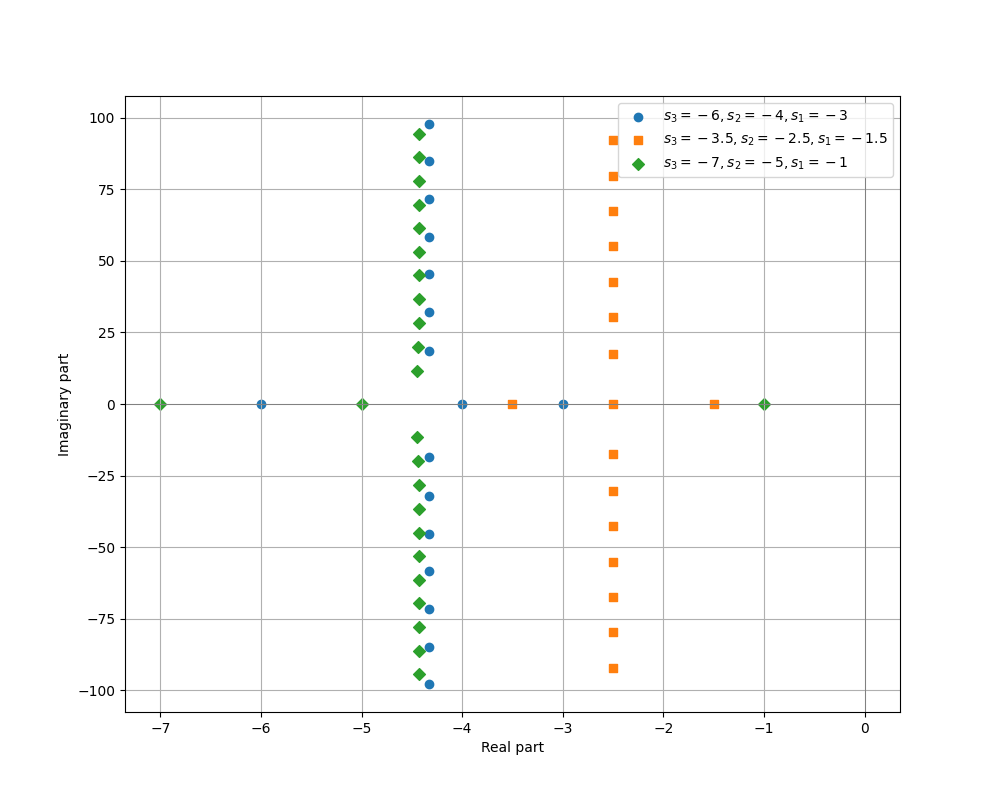}
  \vspace{-0.8cm}
		\caption{Spectrum of the quasipolynomial $\Delta$ for various parameters when $\tau=1$. This is obtained by assigning three real spectral values $s_3<s_2<s_1$ to $\Delta$ and computing each case's system parameters $a$, $\alpha$, and $\beta$. 
  }
		\label{fig:spectrums-Delta-3RR}
\end{figure}

In the specific case where a single real spectral value 
$s_3=s_2=s_1$ is assigned to $\Delta$, the GMID is demonstrated in \cite{boussaada2022generic}. Furthermore, the analysis fully characterizes the remaining spectrum of $\Delta$ as
\begin{equation}\label{eq:spectrum 1 triple rr}
    s = s_1+i\frac{\omega}{\tau},\quad\mbox{where}\quad \tan\left(\frac{\omega}{2}\right) = \frac{\omega}{2}.
\end{equation}

In the following, we fully characterize the spectrum of $\Delta$ when it admits three real spectral values.

\begin{theorem}\label{thm:remaining spectrum general 3RR}
   Assume that $\Delta$ admits three real spectral values $s_3<s_2<s_1$. Then, the remaining spectrum of $\Delta$ is given by
    \begin{equation}\label{eq:spectrum distribution 3rr}
        s = s_2+\frac{\ln(\theta)}{\tau}+i\frac{\omega}{\tau}
    \end{equation}
if and only if
 \begin{equation}\label{eq:xi identity 3rr}
       \ln(\theta) = \xi\frac{(\theta-1)}{2\theta}.
   \end{equation}
The imaginary part $\omega\neq 0$ solves
    \begin{equation}\label{eq:equation imaginary part 3rr}
        \frac{(\theta+1)}{2\theta}\tan\left(\frac{\omega}{2}\right) = \frac{\omega}{\xi}
    \end{equation}
    where $\theta:=\alpha e^{-\tau s_2}$ with $0<\alpha<e^{\tau s_1}$ given by \eqref{eq:alpha} and
    \begin{equation}\label{eq:xi}
        \xi:=\frac{F_{-\tau,1}(s_1,s_3)}{F_{-\tau,2}(s_1,s_2,s_3)}>0.
    \end{equation}
\end{theorem}

As an immediate consequence of Theorem~\ref{thm:remaining spectrum general 3RR}, one has the following.
\begin{corollary}
     Assume that $\Delta$ admits three real spectral values $s_3<s_2<s_1$. If \eqref{eq:xi identity 3rr} is not satisfied, then the remaining spectrum of $\Delta$ forms a chain asymptotic to the vertical line
     \begin{equation}
         \Re(s) = \ln(\alpha)
     \end{equation}
   where $0<\alpha<e^{\tau s_1}$ is given by \eqref{eq:alpha}. More precisely, there exists $s_0\in\C$ such that the following holds
   \begin{equation}
       \Delta(s_0) = 0\quad\mbox{and}\quad\Re(s_0)\neq\ln(\alpha). 
   \end{equation}
\end{corollary}

Before proving Theorem~\ref{thm:remaining spectrum general 3RR}, let us illustrate it in the specific case where the quasipolynomial $\Delta$ has three equidistributed real spectral values.

\begin{corollary}\label{cor:remaining spectrum equidistributed RR}
Assume that $\Delta$ admits three equidistributed real spectral values $s_k = s_1-(k-1)d$, for $d>0$ and $k=1, 2, 3$. Then, the remaining spectrum of $\Delta$ is characterized by
\begin{equation}\label{eq:spectrum 3rr equidistributed}
    s = s_2+i\frac{\omega}{\tau}
\end{equation}
where the imaginary part $\omega\neq 0$ solves
\begin{equation}
    \tan\left(\frac{\omega}{2}\right) = \frac{\omega}{\xi(\tau, d)},\qquad\xi(\tau, d) = \tau d\coth\left(\frac{\tau d}{2}\right).
\end{equation}
Here $\coth$ is the cotangent hyperbolic function.
\end{corollary}

\begin{proof}
    Since $s_3<s_2<s_1$ are equidistributed, one has from $\theta=\alpha e^{-\tau s_2}$, \eqref{eq:alpha} and \eqref{eq:xi} that $\theta = 1$ and $\xi:=\xi(\tau, d) = \tau d\coth\left(\frac{\tau d}{2}\right)$. Since $\theta=1$, and $\xi>0$, then \eqref{eq:xi identity 3rr} is satisfied, and the result follows by \eqref{eq:spectrum distribution 3rr} and \eqref{eq:equation imaginary part 3rr}.
\end{proof}

{\begin{remark}
Since $\xi(\tau, d)\to 2$ as $d\to 0$, \eqref{eq:spectrum 1 triple rr} coincides with \eqref{eq:spectrum 3rr equidistributed}, and Corollary~\ref{cor:remaining spectrum equidistributed RR} confirms the intuitive concept that the GMID property is the limiting case of the GCRRID property. 
  We illustrated in Figure~\ref{fig:function Theta} values of $d:=s_1-s_2$ and $\delta:=s_1-s_3$ where \eqref{eq:xi identity 3rr} is satisfied. We also refer to Figure~\ref{fig:spectrums-Delta-3RR}, where we depicted the quasipolynomial $\Delta$'s spectrum in three scenarios.
\end{remark}}

\begin{proof}[\textit{Proof} of Theorem~\ref{thm:remaining spectrum general 3RR}]
\label{s:remaining spectrum non-equidistributed RR}
   Using \eqref{eq:beta}, \eqref{eq:alpha} and \eqref{eq:a1}, one has
   \begin{equation}
       \alpha =\theta e^{\tau s_2},\qquad  a=-s_2-\frac{\xi}{\tau},\qquad\beta = -\alpha s_2+\frac{\xi}{\tau}e^{\tau s_2}.
   \end{equation}
Let $d:=s_1-2$, $\delta:=s_1-s_3$, and consider the change of functions 
   \begin{equation}\label{eq:Delta prime}
       \Delta'(z) = \tau\Delta\left(s_2+\frac{z}{\tau}\right)=z-\xi+e^{-z}(\theta z+\xi)\qquad (z\in\C)
   \end{equation}
so that $\Delta'$ has three real roots $-\tau(\delta-d)$, $0$ and $\tau d$. In particular, $z\in\C$ is a spectral value of $\Delta'$ if, and only if, $s_2+z/\tau$ is a spectral value of $\Delta$. Moreover, since $s_1$ is a dominant root of $\Delta$, then $\tau d$ is a dominant root of $\Delta'$. Let now $z_0=x_0+i\omega_0$ be a complex non-real spectral value of $\Delta'$. Then, $x_0<\tau d$ and $|\omega_0|\ge 2\pi$ by Lemma~\ref{lem:PS bound}. From $\Delta'(z_0)=0$, one obtains
    \begin{equation}\label{eq:mediatrice non-equidistributed RR}
        \frac{z-\xi}{\theta z+\xi}=-e^{-z}\quad\Longrightarrow\quad \sqrt{\frac{(x_0-\xi)^2+\omega_0^2}{(\theta x_0+\xi)^2+\theta^2\omega_0^2}}= e^{-x_0}.
    \end{equation}
Letting $|\omega_0|\to\infty$ in \eqref{eq:mediatrice non-equidistributed RR}, one deduces that
\begin{equation}
    e^{-x_0} = {1}/{\theta},\qquad\mbox{since}\qquad 0<\theta<e^{\tau(s_1-s_2)}=e^{\tau d}.
\end{equation}
It follows that $x_0 = \ln(\theta)$, and the rest of the spectrum of $\Delta'$ is either located on the vertical line $\Re(z)=\ln(\theta)$ or forms a chain asymptotic to $\Re(z)=\ln(\theta)$. But, if \eqref{eq:xi identity 3rr} is satisfied, then \eqref{eq:Delta prime} yields
\begin{eqnarray}
    \hspace{-0.6cm}\Delta'(2\ln(\theta)-\overline{z}_0)\hspace{-0.3cm}&=&\hspace{-0.3cm}2\ln(\theta)-\overline{z}_0-\xi+\frac{e^{\overline{z}_0}}{\theta^2}\left(\theta(2\ln(\theta)-\overline{z}_0)+\xi\right)\nonumber\\
    &=&\hspace{-0.3cm}-\overline{z}_0-\frac{\xi}{\theta}+\frac{e^{\overline{z}_0}}{\theta}\left(-\overline{z}_0+\xi\right)=-\frac{e^{\overline{z}_0}}{\theta}\Delta'(\overline{z}_0)
\end{eqnarray}
for every $z_0\in\C$. Therefore, since $\Delta'$ has real coefficients, a complex number $\overline{z_0}$ is a spectral value of $\Delta'$ if, and only if, $2\ln(\theta)-\overline{z_0}$, the reflection of $z_0$ across the vertical line $\Re(z)=\ln(\theta)$, is a spectral value of $\Delta'$. As a result, the remaining spectrum of $\Delta'$ exists along the vertical line $\Re(z)=\ln(\theta)$. One deduces that the remaining spectrum of $\Delta$ exists along the vertical line $\Re(s)=s_2+\ln(\theta)/\tau$. Substituting $z = \ln(\theta)+i\omega$ for some real number $\omega\neq 0$ into $\Delta'(z)=0$, yields
\begin{equation}
    \ln(\theta)+i\omega = \xi\frac{(\theta-1)}{2\theta}+i\xi\frac{(\theta+1)}{2\theta}\tan\left(\frac{\omega}{2}\right)
\end{equation}
after careful computations. Then, \eqref{eq:xi identity 3rr} and \eqref{eq:equation imaginary part 3rr} follow, completing the proof of the theorem.
\end{proof}

\begin{figure}[t]
		 \hspace{-0.6cm}
  \includegraphics[width=1.25\linewidth]{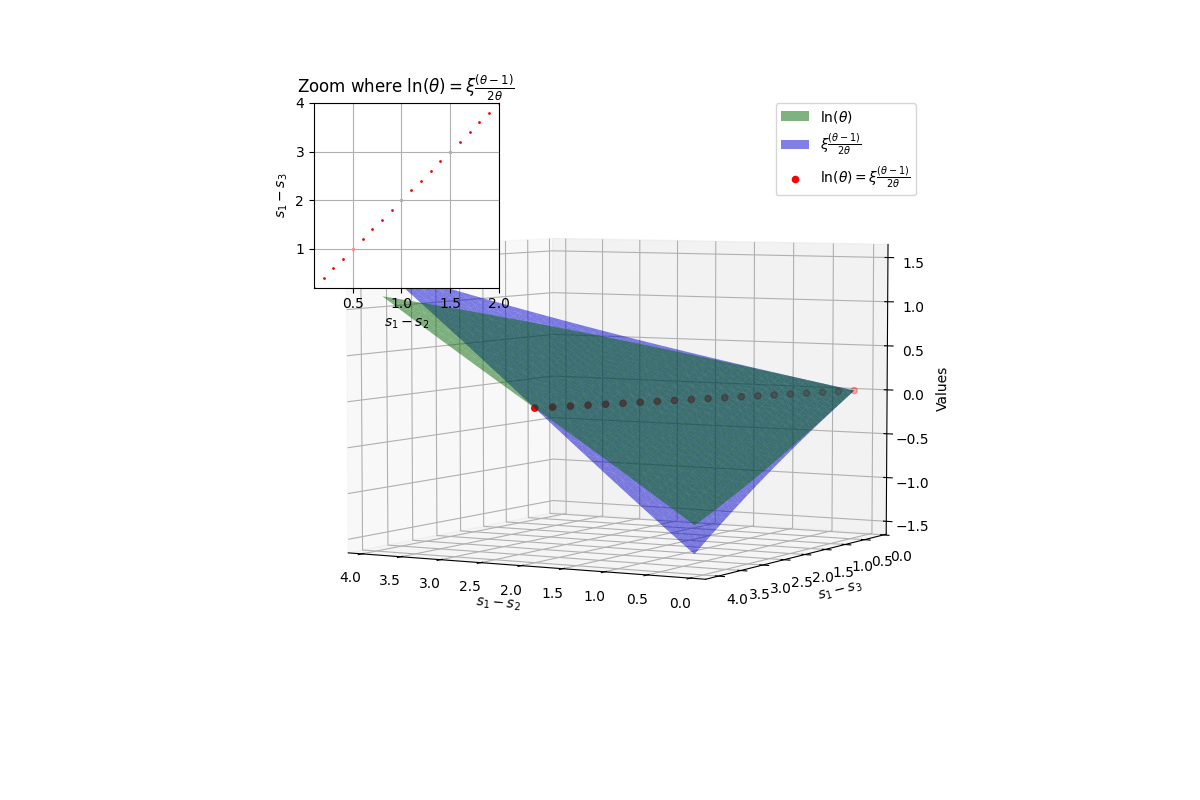}
  \vspace{-2cm}
		\caption{
        {Plot of $\ln(\theta)$ and $\xi\frac{(\theta-1)}{2\theta}$ as functions of $d:=s_1-s_2$ and $\delta:=s_1-s_3$. Here $\theta:=\alpha e^{-\tau s_2}$ with $0<\alpha<e^{\tau s_1}$ given by \eqref{eq:alpha} and $\xi$ is defined by \eqref{eq:xi}. The \textit{red} dots characterize values of $d$ and $\delta$ where \eqref{eq:xi identity 3rr} is satisfied.}
  }
		\label{fig:function Theta}
\end{figure}

\subsection{Assigning two distinct real spectral values}\label{ss:2 DRSV}

This section discusses the  ICRRID property of the quasipolynomial $\Delta$. The aim is to assign two real spectral values $s_2<s_1$ to 
$\Delta$ and to find necessary and sufficient conditions on the coefficients of $\Delta$ guaranteeing that $s_1$ is a dominant root.  Finally, we explicitly characterize the spectrum of $\Delta$ when it admits exactly two real roots $s_2<s_1$. 

{
Inspired from \cite[Theorem~6]{schmoderer2023boundary}, we find these conditions in terms of $(a+s_1)/(s_1-s_2)$. The challenge relies on the scenario where a third real spectral value $x$ coexists in the spectrum of $\Delta$, in which case the coefficients $a$, $\alpha$, and $\beta$ are given in Lemma~\ref{lem:coeffs}. Therefore, the greatest between $s_2$, $s_1$, and $x$ is a dominant root by Theorem~\ref{thm:dominance of a real root 3 rr}. Since one wants $s_1$ to be a dominant root of $\Delta$, we are looking for the necessary and sufficient conditions on $(a+s_1)/(s_1-s_2)$ ensuring that $x<s_2<s_1$ or $s_2<x<s_1$. To this end, one deduces from \eqref{eq:a1} that 
\begin{equation}
    \frac{a+s_1}{s_1-s_2} = -\frac{(s_1-x)(e^{\tau s_2}-e^{\tau x})}{e^{\tau x}(s_1-s_2)+e^{\tau s_2}(x-s_1)+e^{\tau s_1}(s_2-x)}
\end{equation}
for $\forall x\in\R\bs\{s_1,s_2\}$, which is negative owing to \eqref{eq:recursive function} and \eqref{eq:recursive function 1}. On the other hand, introduce for all $x\in\R\bs\{s_1,s_2\}$, 
\begin{equation}
    \varphi(x) = \frac{(s_1-x)(e^{\tau s_2}-e^{\tau x})}{e^{\tau x}(s_1-s_2)+e^{\tau s_2}(x-s_1)+e^{\tau s_1}(s_2-x)}.
\end{equation}
Then, $\varphi(x)>0$ for every $x<s_2<s_1$ or $s_2<x<s_1$ or $x>s_1$, and satisfies
 \begin{equation*}
      \begin{split}
 	\lim_{x\to-\infty}\varphi(x)&=\frac{1}{e^{\tau(s_1-s_2)}-1}\\
\lim_{x\to s_2}\varphi(x)&=\frac{\tau(s_1-s_2)}{e^{\tau(s_1-s_2)}-(1+\tau(s_1-s_2))}\\
\lim_{x\to s_1}\varphi(x)&=\frac{e^{\tau(s_1-s_2)}-1}{e^{\tau(s_1-s_2)}(\tau(s_1-s_2)-1)+1}.
  \end{split}
 \end{equation*}
Moreover, letting for all $x\in\R\bs\{s_1,s_2\}$,
\begin{equation*}
\begin{alignedat}{2}
    \psi(x)&=\tau^2(s_2-x)(s_1-x)(s_1-s_2)(F_{-\tau,1}(x,s_1)F_{-\tau,1}(x,s_2)\\
    &\quad -e^{x}F_{-\tau,1}(s_1,s_2))
\end{alignedat} 
\end{equation*}
one finds
\begin{equation*}
    \varphi'(x)=\frac{\psi(x)}{[e^{x \tau} (-s_1 + s_2) + 
e^{s_2 \tau} (s_1 - x) + e^{s_1 \tau} (-s_2 + x)]^2}
\end{equation*}
thanks to the integral representation \eqref{eq:F 1}. Using \cite[Lemma~7]{bedouhene2020real}, one finds that
\[
F_{-\tau,1}(u,v)-F_{-\tau,1}(u,w) = \tau(v-w)F_{-\tau,2}(u,v,w)>0\Leftrightarrow v>w
\]
for all $\tau>0$ and $(u, v, w)\in\R^3$. Therefore, for a fixed $u\in\R$, $v\mapsto F_{-\tau,1}(u,v)$ is positive and increasing  on $\R$. One deduces that $\varphi'(x)>0$ for every $x<s_2<s_1$ or $s_2<x<s_1$ or $x>s_1$, and that $\varphi$ is a positive increasing function for $x<s_2<s_1$ or $s_2<x<s_1$ or $x>s_1$. Therefore, if $s_2<s_1$, then
\[
\varphi(x)\ge \frac{1}{e^{\tau(s_1-s_2)}-1},\qquad (\forall x\in\R\bs\{s_1,s_2\}).
\]

We proved the following.
\begin{theorem} \label{thm:condition of the domiancy 2 rr with 1RR suplementary}
    Assume that the quasipolynomial $\Delta$ admits two real spectral values $s_2 < s_1$. Then, a third real spectral value $x$ coexists in the spectrum of $\Delta$ if and only if
    \begin{equation}
      \Lambda_3 :=  \frac{a + s_1}{s_1 - s_2} \le \frac{-1}{e^{\tau(s_1 - s_2)} - 1}.
    \end{equation}
    Furthermore,
    \begin{enumerate}
        \item $x < s_2 < s_1$ if and only if
        \begin{equation*}\label{eq:condition of the domiancy 2 rr with 1RR suplementary smaller}
        \frac{-\tau(s_1 - s_2)}{e^{\tau(s_1 - s_2)} - (1 + \tau(s_1 - s_2))} < \Lambda_3 \le \frac{-1}{e^{\tau(s_1 - s_2)} - 1}
        \end{equation*}
        \item $s_2 < x < s_1$ if and only if
        \begin{equation*}\label{eq:condition of the domiancy 2 rr with 1RR suplementary medium}
        \hspace{-0.5cm}\frac{1 - e^{\tau(s_1 - s_2)}}{e^{\tau(s_1 - s_2)} (\tau(s_1 - s_2) - 1) + 1} < \Lambda_3 < \frac{-\tau(s_1 - s_2)}{e^{\tau(s_1 - s_2)} - (1 + \tau(s_1 - s_2))}
        \end{equation*}
        \item $s_2 < s_1 < x$ if and only if
        \begin{equation*}\label{eq:condition of the domiancy 2 rr with 1RR suplementary greater}
        \Lambda_3 \le \frac{1 - e^{\tau(s_1 - s_2)}}{e^{\tau(s_1 - s_2)} (\tau(s_1 - s_2) - 1) + 1}.
        \end{equation*}
    \end{enumerate}
\end{theorem}

Let us now investigate the necessary and sufficient conditions on $(a+s_1)/(s_1-s_2)$ guaranteeing the dominance of a simple real spectral value $s_1$ when exactly two real spectral values $s_2<s_1$ are assigned to the quasipolynomial $\Delta$.
\begin{lemma}\label{lem:ICRRID}
    The quasipolynomial $\Delta$ admits exactly two distinct real spectral values $s_2$ and $s_1$ if, and only if $a\in\R$,
    \begin{equation}\label{eq:beta 2 rr}
        \beta = -\alpha s_1-e^{\tau s_1}(a+s_1)=-\alpha s_2-e^{\tau s_2}(a+s_2).
    \end{equation}
    \begin{equation}\label{eq:a and alpha}
        \alpha = \frac{-(a+s_1)e^{\tau s_1}+(a+s_2)e^{\tau s_2}}{s_1-s_2}.
    \end{equation}
\end{lemma}

\begin{remark}\label{rmk:ICCRID exactly 2RR}
  It is an immediate consequence of Theorem~\ref{thm:condition of the domiancy 2 rr with 1RR suplementary} that, exactly two real spectral values $s_2<s_1$ coexist in the spectrum of $\Delta$, if and only if, the following inequality holds
\begin{equation}\label{eq:a priori inequalities}
   \frac{-1}{e^{\tau(s_1-s_2)}-1}< \frac{a+s_1}{s_1-s_2}.
\end{equation}

\end{remark}
The next lemma is a key ingredient in proving our second main result.
\begin{lemma}\label{lem:properties on alpha and beta ICRRID}
Assume that $\Delta$ admits exactly two real spectral values $s_2 < s_1$. Then, the following conditions hold.
\begin{equation}\label{eq:condition of the domiancy 2 rr}
  \frac{-1}{e^{\tau(s_1-s_2)} - 1} < \frac{a + s_1}{s_1 - s_2} < 1 \quad \Longleftrightarrow  \quad  -e^{\tau s_1} < \alpha < 0.
\end{equation}
In particular, the necessary part of \eqref{eq:condition of the domiancy 2 rr} implies that
\begin{equation}\label{eq:asymptotic alpha 2RR}
    -e^{\tau x}<\alpha<0 \qquad \forall x \ge s_1.
\end{equation}
Additionally, if $s_1 \le 0$, then $\beta \le 0$ whenever
\begin{equation}\label{eq:asymptotic beta 2RR}
    0 \le \frac{a + s_1}{s_1 - s_2} < 1.
\end{equation}
\end{lemma}


\begin{proof}
 Firstly, from \eqref{eq:a and alpha}, one gets
 \begin{equation}
     \alpha = -\frac{a+s_1}{s_1-s_2}(e^{\tau s_1}-e^{\tau s_2})-e^{\tau s_2}.
 \end{equation}
 Therefore, the equivalence \eqref{eq:condition of the domiancy 2 rr} holds owing to
 \begin{equation*}
      -(e^{\tau s_1}-e^{\tau s_2})-e^{\tau s_2}<\alpha<\frac{e^{\tau s_1}-e^{\tau s_2}}{e^{\tau(s_1-s_2)}-1}-e^{\tau s_2}\Leftrightarrow -e^{\tau s_1}<\alpha<0.
 \end{equation*}
 Inequality \eqref{eq:asymptotic alpha 2RR} is a consequence of the necessary part of \eqref{eq:condition of the domiancy 2 rr} and that $x\in\R\mapsto e^{x}$ is increasing. 
Finally, if $s_1\le 0$, then $\beta\le 0$ follows immediately from \eqref{eq:beta 2 rr}, $\alpha<0$ and \eqref{eq:asymptotic beta 2RR}.
\end{proof}
}

One can now prove the following result, which gives necessary and sufficient conditions for the dominance of a real spectral value when $\Delta$ admits exactly two real spectral values.
\begin{theorem}\label{thm:dominance of a real root 2 rr}
Assume that exactly two real spectral values $s_2<s_1$ coexist in the spectrum of $\Delta$. Then, $s_1$ is a strictly dominant root of $\Delta$ if, and only if, either the sufficient or the necessary part of \eqref{eq:condition of the domiancy 2 rr}  are satisfied.
\end{theorem}

\begin{proof}
Firstly, it follows from Remark~\ref{rmk:ICCRID exactly 2RR} that $\Delta$ admits exactly two real spectral values if, and only if, the sufficient part in \eqref{eq:condition of the domiancy 2 rr} is satisfied. 
Let $s_0:=x+i\omega\in\C$ be such that $\Delta(s_0) = 0$. In particular, $\omega\neq 0$. From $\Delta(s_1)=0$, one gets $\beta = -\alpha s_1-e^{\tau s_1}(a+s_1)$, where $\alpha$ and $a$ are given by Lemma~\ref{lem:ICRRID}. It follows that $ \Delta(s_0) = 0$ if, and only if,
\begin{equation}\label{eq:delta s 0 vanishes 2 sv}
    e^{\tau s_0}s_0+e^{\tau s_0}a+\alpha s_0 = \alpha s_1+e^{\tau s_1}(a+s_1).
\end{equation}
By taking the real and imaginary parts of both sides in \eqref{eq:delta s 0 vanishes 2 sv}, one gets
 \begin{equation}\label{eq:s 0 and its cc 2 sv}
      \begin{split}
 	e^{\tau x}((a+x)\cos(\tau \omega)-\omega\sin(\tau\omega))&=e^{\tau s_1}(a+s_1)-\alpha(x-s_1)\\
e^{\tau x}((a+x)\sin(\tau \omega)+\omega\cos(\tau\omega))&=-\alpha\omega.
  \end{split}
 \end{equation}
 By squaring each equality in \eqref{eq:s 0 and its cc 2 sv} and adding them, one obtains
 \begin{equation}
 \omega^2=\frac{(e^{\tau s_1}(a+s_1)-\alpha(x-s_1))^2-(a+x)^2e^{2\tau x}}{e^{2\tau x}-\alpha^2},
 \end{equation}
  which is well-defined for every $x\ge s_1$ by Lemma~\ref{lem:properties on alpha and beta ICRRID}. Therefore, by performing the exact same steps of the proof of dominance in Theorem~\ref{thm:dominance of a real root 3 rr}, one gets that $\omega^2<1/\tau^2$, i.e., $|\omega|<1/\tau$. The latter is inconsistent since one has necessarily $|\omega|\ge \pi/\tau$ owing to Lemma~\ref{lem:PS bound}. Hence, $s_1$ is a dominant root of $\Delta$. Conversely, assume that either the sufficient or the necessary part of \eqref{eq:condition of the domiancy 2 rr} is not satisfied. 
  Then, another real root $x$ coexists in the spectrum of $\Delta$ by Theorem~\ref{thm:condition of the domiancy 2 rr with 1RR suplementary} or
   $s_1$ and $s_2$ are the only real spectral values of $\Delta$, and $s_1$ is not a strictly dominant root of $\Delta$ by Lemma~\ref{lem:properties on alpha and beta ICRRID} and the sufficient part of this theorem. 
\end{proof}

\begin{figure}[t]
		\centering
		\includegraphics[width=\linewidth]{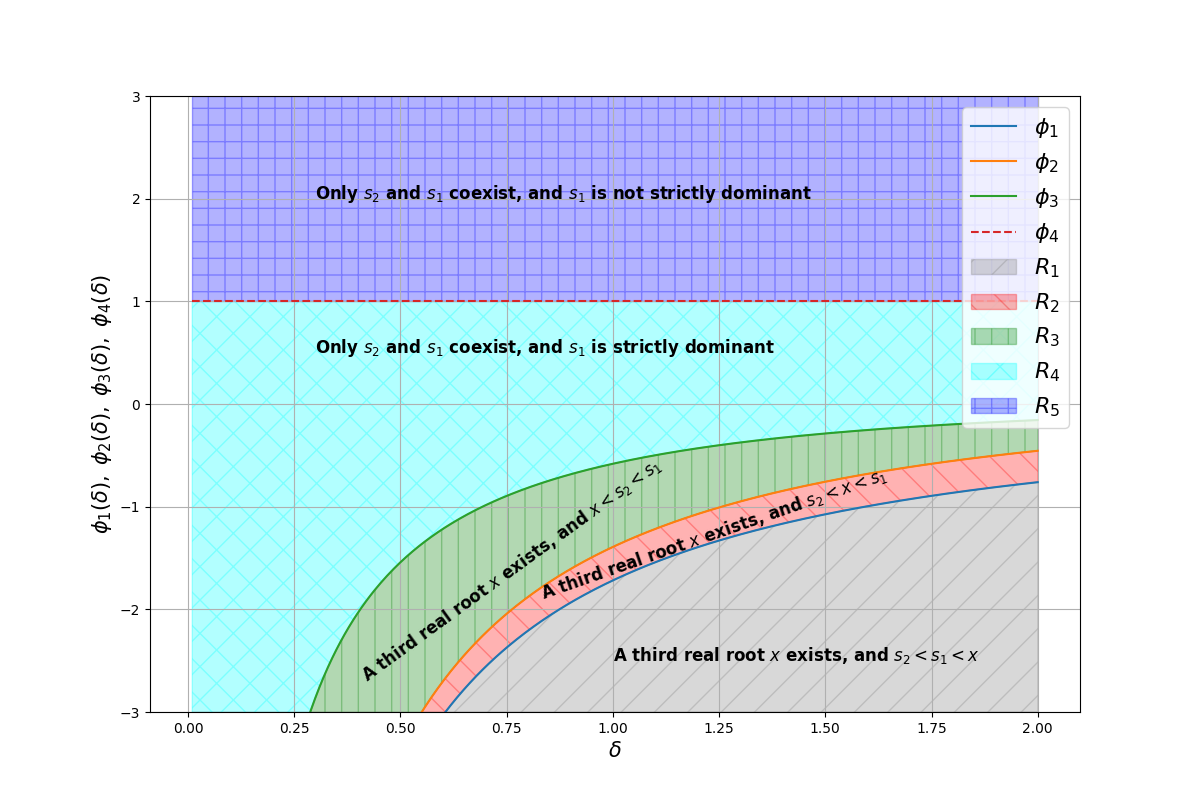}
  \vspace{-0.5cm}
		\caption{Complete characterization of the regions with respect to the values of $(a+s_1)/\delta$ that are necessary and sufficient to ensure the strict dominance or not of the real spectral value $s_1$ when only two real spectral values $s_2<s_1$ are assigned to $\Delta$. The figure is depicted when the delay $\tau=1$ while $\delta:=s_1-s_2$ ranges between $0.01$ and $2$. The functions used are $\phi_1(\delta)=(1-e^{\delta})/(e^{\delta}(\delta-1)+1)$, $\phi_2(\delta) = -\delta/(e^{\delta}-(1+\delta))$, $\phi_3(\delta) = -1/(e^{\delta}-1)$, and $\phi_4(\delta) = 1$. We define the regions $R_1$ as the range of $\delta$ where $\frac{a+s_1}{\delta}\le\phi_1(\delta)$, $R_2$ as the range where  $\phi_1(\delta)<\frac{a+s_1}{\delta}\le\phi_2(\delta)$, $R_3$ as the range where  $\phi_2(\delta)<\frac{a+s_1}{\delta}\le\phi_3(\delta)$, $R_4$ as the range where  $\phi_3(\delta)<\frac{a+s_1}{\delta}<\phi_4(\delta)$, and $R_5$ as the range where  $\frac{a+s_1}{\delta}\ge\phi_4(\delta)$.}
		\label{fig:comparison-region-2RR}
\end{figure}
\begin{remark}\label{rmk:dominance and not strict 2rr}
    In the particular case of exactly two real spectral values $s_2<s_1$ coexisting in the spectrum of $\Delta$, and $\alpha=-e^{\tau s_1}$ then $s_1$ is a dominant (not strict) root of $\Delta$ since its remaining spectrum is given by $s=s_1+i\frac{2\pi k}{\tau}$,\;$k\in\Z$. 
    Indeed, if $\alpha=-e^{\tau s_1}$, one has
    $\Delta(s) = (s-s_2)(1-e^{-\tau(s-s_1)})$ for every $s\in\C$, from which one obtains the desired result. When $\alpha=0$, $s_1$ is still a strictly dominant root of $\Delta$. Indeed, it is a dominant root by \cite[Theorem~12]{bedouhene2020real}. Next, one finds in this case
    \begin{equation*}
        \Delta(s) = \frac{e^{\tau (s_1-s_2)}(s-s_1)+(s_1-s_2)(e^{-\tau(s-s_1)}-1)}{e^{\tau(s_1-s_2)}-1}.
    \end{equation*}
    Letting $s=s_1+i\omega$, one deduces that $\Delta(s)=0$ if and only if
    \begin{equation*}
         \cos(\tau\omega) = 1,\quad \omega e^{\tau (s_1-s_2)}=(s_1-s_2)\sin(\tau\omega).
    \end{equation*}
    where from one obtains $\omega=0$. It follows that a non-real root $s$ of $\Delta$ necessary satisfies $\Re(s)<s_1$.
\end{remark}
Aside from the necessary and sufficient conditions on the dominance of the assigned rightmost root $s_1$, which have already been established in \cite[Theorem~6]{schmoderer2023boundary},   Theorem \ref{thm:condition of the domiancy 2 rr with 1RR suplementary} and Theorem~\ref{thm:dominance of a real root 2 rr} provide a complete description of such conditions with respect to the number of assigned roots and the potential coexistence of a third real root. 
We refer to Figure~\ref{fig:comparison-region-2RR} for a full visualization.
The following corollary is an immediate consequence of Theorem~\ref{thm:dominance of a real root 2 rr} and Remark~\ref{rmk:dominance and not strict 2rr}.  It generalizes and gives a straightforward proof of the IMID property provided in \cite[Theorem~3.1-\textit{item~2}]{benarab2023multiplicity}.
\begin{corollary}\label{cor:condition of the dominance 1rr with multiplicity 2}
      Assume that $\Delta$ admits two real spectral values $s_2<s_1$. Then, $s_1$ is a dominant root of $\Delta$ of multiplicity equal to two if and only if
      \begin{equation}\label{eq:condition of the dominance 1rr with multiplicity 2}
          -1\le\tau(a+s_1)\le 0\qquad (\forall a\in\R).
      \end{equation}
      In particular, $s_1$ is a strictly dominant root of $\Delta$ if and only if the right inequality in \eqref{eq:condition of the dominance 1rr with multiplicity 2} is strict.
\end{corollary}
\begin{proof}
    Assume that $\Delta$ admits two real spectral values $s_2<s_1$. By Theorem~\ref{thm:dominance of a real root 2 rr} and Remark~\ref{rmk:dominance and not strict 2rr}, the real root $s_1$ is a dominant root of $\Delta$ if and only if $-e^{\tau s_1}\le\alpha\le 0$. If in addition $\Delta'(s)=0$, then $\alpha=-(\tau(a+s_1)+1)e^{\tau s_1}$. The result follows.
\end{proof}

As in Theorem~\ref{thm:remaining spectrum general 3RR}, let us end this section by fully characterizing the remaining spectrum of $\Delta$ when exactly two real roots coexist in its spectrum. The proof follows \emph{mutatis mutandis} that of Theorem~\ref{thm:remaining spectrum general 3RR}.
\begin{theorem}\label{thm:remaining spectrum general 2RR}
    Assume that $\Delta$ has exactly two real spectral values $s_2<s_1$. Then, its remaining spectrum is given by 
    \begin{equation}
        s = s_2+\frac{\ln(-\theta)}{\tau}+i\frac{\omega}{\tau}
    \end{equation}
   if and only if
 \begin{equation}\label{eq:xi identity 2rr}
       \ln(-\theta) = \xi\frac{(\theta-1)}{2\theta}.
   \end{equation}
The imaginary part $\omega\neq 0$ solves
    \begin{equation}\label{eq:equation imaginary part 2rr}
        \frac{(\theta+1)}{2\theta}\cot\left(\frac{\omega}{2}\right) = -\frac{\omega}{\xi}
    \end{equation}
    where $\theta:=\alpha e^{-\tau s_2}$ with $-e^{\tau s_1}<\alpha<0$ given by \eqref{eq:a and alpha} and 
    \begin{equation}
        \xi = \frac{(\alpha+e^{\tau s_1})}{F_{-1,\tau}(s_1,s_2)}>0.
    \end{equation}
\end{theorem}
As an immediate consequence of Theorem~\ref{thm:remaining spectrum general 2RR}, one has the following.
\begin{corollary}
     Assume that $\Delta$ admits exactly two real spectral values $s_2<s_1$. If \eqref{eq:xi identity 2rr} is not satisfied, then the remaining spectrum of $\Delta$ forms a chain asymptotic to the vertical line
     \begin{equation}
         \Re(s) = \ln(-\alpha) 
     \end{equation}
   where $-e^{\tau s_1}<\alpha<0$ is given by \eqref{eq:a and alpha}. More precisely, there exists $s_0\in\C$ such that the following holds
   \begin{equation}
       \Delta(s_0) = 0\quad\mbox{and}\quad\Re(s_0)\neq\ln(-\alpha). 
   \end{equation}
\end{corollary}

\subsection{On Frasson-Verduyn Lunel's sufficient conditions for dominance and beyond}\label{ss:Frasson's sufficient conditions}

This section builds upon Frasson-Verduyn Lunel's seminal work \cite[Lemma~5.1]{frasson2003large}, which established a sufficient condition for the dominance of a simple real spectral value of quasipolynomials with multiple delays. While Frasson-Verduyn Lunel's lemma offers a fundamental method for assessing spectral dominance, its applicability is restricted to specific conditions that may only address certain dynamic scenarios encountered in complex systems. Restricted to the single-delay case, when Frasson-Verduyn Lunel's condition is not met, the CRRID property still provides a guarantee for the dominance of a simple real root.

The first result of this section states as follows.

\begin{theorem}\label{thm:beyond Frasson's Lemma GCRRID}
Let $\tau>0$. There exist $s_1\le 0$ and $d>0$, such that if $\Delta$ admits three equidistributed real spectral values $s_3=s_1-2d$, $s_2=s_1-d$ and $s_1$, then
\begin{equation}
    V(s_1) = (|\alpha|(1+|s_1|\tau)+|\beta|\tau)e^{-s_1\tau}
\end{equation}
satisfies $V(s_1)\ge 1$ and $s_1$ is a dominant root of $\Delta$.
\end{theorem}

\begin{proof}
    Assume that the quasipolynomial $\Delta$ admits three equidistributed real spectral values $s_k = s_1-(k-1)d$, for some $d>0$ (that we will choose later on) and $k=1, 2, 3$ with $s_1\le 0$. Necessarily, $s_k$ is simple since $\Delta$ cannot admit more than three real roots counting multiplicities  \cite[Problem~206.2, page~144]{polya1972analysis}. Hence, Theorem~\ref{thm:dominance of a real root 3 rr} states that $s_1$ is a strictly dominant root of $\Delta$. Moreover, one deduces from Lemma~\ref{lem:coeffs} that
     \begin{equation*}
         \alpha  =e^{\tau(s_1-d)},\quad \beta = e^{\tau(s_1-d)}\left(-s_1+d+d\coth\left({\tau d}/{2}\right)\right),
    \end{equation*}
    \begin{equation}\label{eq:function V}
        V(s_1) = \left(1-2\tau s_1+\tau d+\tau d\coth\left({\tau d}/{2}\right)\right)e^{-\tau d}.
    \end{equation}
    Introducing $g(s_1):=V(s_1)-1$, one finds
    \begin{equation*}
        g(s_1) = e^{-\tau d}\left(\tau d\coth\left({\tau d}/{2}\right)+\tau d+1-e^{\tau d}-2\tau s_1\right).
    \end{equation*}
It follows that $g(s_1)\ge 0$ if, and only if,
    \begin{equation}\label{eq:valid s1 for non Frasson}
        s_1\le({\tau d\coth\left({\tau d}/{2}\right)+\tau d+1-e^{\tau d}})/{2\tau}.
    \end{equation}

In particular, knowing that Frasson-Verduyn Lunel's sufficient condition is not met the first time when $V(s_1) = 1$, owing to inequality \eqref{eq:valid s1 for non Frasson} the latter is equivalent to ($v:=\tau s_1$ and $u:=\tau d$)
\begin{equation}\label{eq:first valid s1 for non Frasson}
    v(u) = ({u\coth\left({u}/{2}\right)+u+1-e^{u}})/{2}\qquad (u>0).
\end{equation}

A straightforward analysis of the function $u\mapsto v(u)$ for $u>0$ shows that $v$ is strictly decreasing on $(0,\infty)$ and satisfies
\begin{equation}
     \lim_{u\to 0}v(u) =1,\qquad\qquad\lim_{u\to\infty}v(u)  =-\infty.
\end{equation}
Therefore, there exist $s_1\le 0$ and $d>0$ satisfying \eqref{eq:valid s1 for non Frasson} such that $V(s_1)\ge 1$.
\end{proof}

Theorem~\ref{thm:beyond Frasson's Lemma GCRRID} shows that the sufficient condition of Lemma~\ref{lem:Frasson-Verduyn Lunel} is not necessary for the dominance of simple real spectral values when $\Delta$ admits three equidistributed real roots.

The function $V$ defined in \eqref{eq:function V} initially seems to depend on three separate parameters: the delay $\tau$, the distance $d$, and a dominant simple real value $s_1$. However, upon closer inspection, it becomes apparent that $V$ can be expressed solely in terms of the products $\tau d$ and $\tau s_1$, denoted by $u$ and $v$, respectively. As a result, the function reduces to 
\begin{equation}\label{eq:function V-1} 
W(u,v):= \left(1-2v+u+u\coth\left(\frac{u}{2}\right)\right)e^{-u}. 
\end{equation} 

For visualization, it is sufficient to consider plots of $W$ as a function of $u=\tau d$ and $v=\tau s_1$, treating $\tau$ as a positive constant scaling factor. Refer to Figure~\ref{fig:heatmap-Z-W} (image on the \textit{left}).

The second main result of this section is the following.

\begin{theorem}\label{thm:beyond Frasson's Lemma ICRRID}
Let $\tau>0$. There exist $s_1\le 0$ and $\delta>0$ such that if $\Delta$ admits exactly two simple real spectral values $s_1$ and $s_2=s_1-\delta$, and inequality \eqref{eq:asymptotic beta 2RR} is satisfied, then
\begin{equation}
    Y(s_1) = (|\alpha|(1+|s_1|\tau)+|\beta|\tau)e^{-s_1\tau}
\end{equation}
satisfies $Y(s_1)\ge 1$ and $s_1$ is a dominant root of $\Delta$.
\end{theorem}

\begin{proof}
    If $\Delta$ admits exactly two simple real spectral values $s_2<s_1$ and inequality \eqref{eq:asymptotic beta 2RR} is satisfied, then $s_1$ is a strictly dominant root of $\Delta$ by Theorem~\ref{thm:dominance of a real root 2 rr}. Moreover, $\alpha<0$, and if $s_1\le 0$ then $\beta\le 0$ by Lemma~\ref{lem:properties on alpha and beta ICRRID}. If $s_2=s_1-\delta$ for some $\delta>0$ that will be chosen later on, then one has from Lemma~\ref{lem:ICRRID},
    \begin{equation*}
    \begin{split}
        \alpha  &= -e^{\tau s_1}\left(\frac{(a+s_1)}{\delta}(1-e^{-\tau\delta})+e^{-\tau\delta}\right)\\
        \beta &= -\alpha s_1-e^{\tau s_1}(a+s_1)
    \end{split}
    \end{equation*}
     so that
     \begin{equation*}
         \hspace{-0.2cm}Y(s_1)=\frac{e^{-\tau\delta}(a+s_1-\delta)(-1+2\tau s_1)+(a+s_1)(1-2\tau s_1+\tau\delta)}{\delta}.
     \end{equation*}
    Upon closer inspection, one finds that $Y$ can be expressed solely in terms of the products $\tau a$, $\tau \delta$, and $\tau s_1$. Setting $A:=\tau a$, $u:=\tau\delta$, $v:=\tau s_1$, then $A\in\R$, $u>0$ and $v<0$. Introducing the functions
    \begin{equation}\label{eq:function Z and h}
        Z(A,u,v):=Y(s_1),\quad h(A,u,v) = Z(A,u,v)-1
    \end{equation}
    one gets
    \begin{equation*}
    \begin{alignedat}{2}
        {u}h(A,u,v) &= -(2-2e^{-u})v^2-((1-e^{-u})(2A-u-1)+ue^{-u})v\\
        &\quad +(A+e^{-u}-1)u+A(1-e^{-u})
    \end{alignedat}
    \end{equation*}
which is a second-degree polynomial in $v$. Let $N(A,u,v)$ be the numerator of $h(A,u,v)$. From $N(A,u,v)=0$, one finds that its discriminant
$D_1(A,u)$ admits two real roots $A_0(u)<0$ for every $u>0$ and $A_1(u)$ that can change sign for $u>0$.
One infers that for every $u>0$, $D_1(A,u)\ge 0$ for $A\le A_0(u)$ and $A\ge A_1(u)$. For simplicity, assume from now on that $A\ge A_1(u)$. Hence, equation  $N(A,u,v)=0$ admits two real roots
\begin{equation}
    \begin{split}
        v_1(A,u)&=\frac{-C_b(A,u)-\sqrt{D_1(A,u)}}{2C_a(u)}\\
        v_2(A,u)&=\frac{-C_b(A,u)+\sqrt{D_1(A,u)}}{2C_a(u)}.
    \end{split}
\end{equation}
 Here
\[
   C_a(u) = -2(1-e^{-u}),\quad C_c(A,u)=(1-e^{-u}+u)A-(1-e^{-u})u,
\]
\[
C_b(A,u)=-2(1-e^{-u})A+(1+u)(1-e^{-u})-ue^{-u}
\]
are the coefficients of $N(A,u,v)$ considered as a second-degree polynomial in $v$, so that $D_1(A,u)=C_b^2(A,u)-4C_a(u)C_c(A,u)$. Letting
\begin{equation*}
    A_2(u)=\frac{(1-e^{-u})u}{1-e^{-u}+u},\quad A_3(u) = \frac{(1+u)(1-e^{-u})-ue^{-u}}{2(1-e^{-u})}
\end{equation*}
one can checks that for all $u>0$, it holds $A_2(u)>0$, $A_3(u)>0$ and $A_1(u)\le A_2(u)\le A_3(u)$. Moreover, one gets the following sign tab.
\begin{center}
\begin{tikzpicture}
\tkzTabInit[lgt=2, espcl=1.5]{$A$ / 1 , {${C_b(A,u)}$} / 0.8, {${C_c(A,u)}$} / 0.8}{$-\infty$, $A_2(u)$, $A_3(u)$, $+\infty$}
\tkzTabLine{, +, t, +, z, -, }
\tkzTabLine{, -, z, +, t, +, }
\end{tikzpicture}
\end{center}
Hence, for all $A\le A_2(u)$, one has $v_1\le 0$ and $v_2> 0$. Since $A_1(u)\le A_2(u)$, and $C_a(u)<0$ for every $u>0$, one deduces that for all $A\in[A_1,A_2]$, $u>0$ and $v\in[v_1,v_2]$, one has $N(A,u,v)\ge 0$. Therefore, for every $s_1\le 0$ and every $\delta>0$, one has $Y(s_1)\ge 1$, whenever
\begin{equation}\label{eq:condition on s1 and a}
    \begin{cases}
        A_1(\tau\delta)\le\tau a\le A_2(\tau\delta)\\
        \mbox{and}\\
        v_1(\tau a,\tau\delta)\le\tau  s_1\le v_2(\tau a,\tau\delta).
    \end{cases}\qedhere
\end{equation}
\end{proof}
\begin{figure}
	\centering
	\includegraphics[width=.49\linewidth]{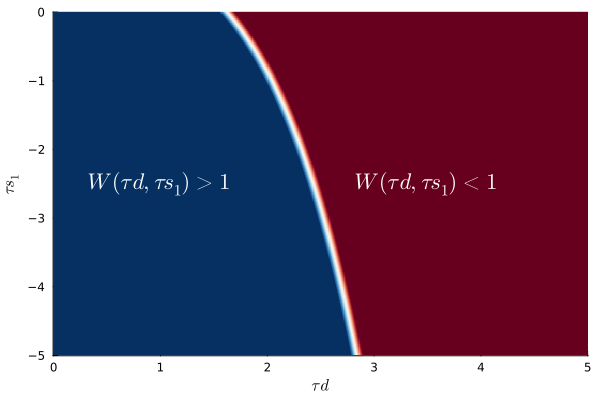}\hspace{0.1em}
	\includegraphics[width=.49\linewidth]{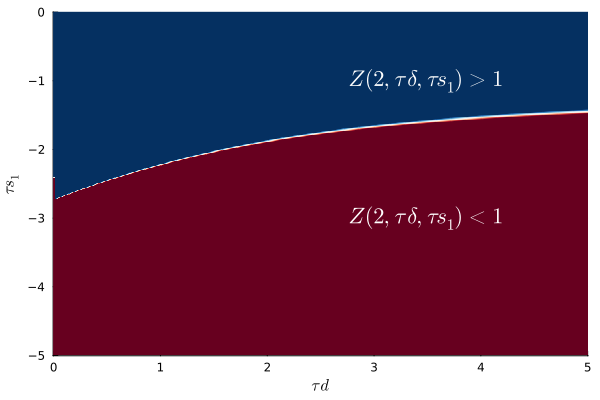}
 \vspace{-0.4cm}
	\caption{Depiction of the functions $W(\tau d, \tau s_1)$ and $Z(2, \tau\delta,\tau s_1)$ defined respectively in \eqref{eq:function V-1} and \eqref{eq:function Z}. The \textit{red}-colored region in the parameters space   $(\tau d, \tau s_1)$  corresponds to the region where Frasson-Verduyn Lunel sufficient condition for the dominance of a simple real spectral value is satisfied. The CRRID property extends the conditions by the \textit{blue} region.}
	\label{fig:heatmap-Z-W}
\end{figure}
Theorem~\ref{thm:beyond Frasson's Lemma ICRRID} shows that the sufficient condition of Lemma~\ref{lem:Frasson-Verduyn Lunel} is not necessary for the dominance of simple real spectral values when $\Delta$ admits exactly two simple real roots.

\begin{remark}
We stress that there are values for $s_1$ and $\delta$ (where $s_1 \le 0$ and $\delta > 0$) that satisfy \eqref{eq:condition on s1 and a} and still maintain inequality \eqref{eq:asymptotic beta 2RR}. From equation \eqref{eq:condition on s1 and a}, take, for instance,
\begin{equation}
    \tau a = A_2(\tau \delta), \qquad \text{and} \qquad \tau s_1 = v_1(\tau a, \tau \delta).
\end{equation}
    One finds, after careful computations, that
   \begin{equation}\label{eq:selected s1 and tau}
       \frac{a+s_1}{\delta} = \frac{A_2(\tau \delta)+v_1(\tau a, \tau \delta)}{\tau\delta} = \frac{1}{2}\left(\frac{2-e^{\tau\delta}}{1-e^{\tau\delta}}+\frac{1}{\tau\delta}\right).
   \end{equation}
  Finally, as a function of $\tau\delta>0$, one immediately checks that the right side in \eqref{eq:selected s1 and tau} is bounded between $1/2$ and $3/4$.
\end{remark}
\begin{remark}
   When two real roots are assigned, and a third root coexists (see, Theorem \ref{thm:condition of the domiancy 2 rr with 1RR suplementary}), one can use the exact same arguments as in the proof of Theorem \ref{thm:beyond Frasson's Lemma GCRRID} to show that in some parameters regions, the Frasson-Verduyn Lunel's criteria does not apply. Furthermore, when exactly two real spectral values $s_2<s_1$ coexist in the spectrum of $\Delta$, and 
   \begin{equation}
       \frac{-1}{e^{\tau(s_1-s_2)}-1}<\frac{a+s_1}{s_1-s_2}\le 0,
   \end{equation}
   then $s_1$ is a strictly dominant root of $\Delta$ by Theorem~\ref{thm:dominance of a real root 2 rr}, and we can prove as in Theorem~\ref{thm:beyond Frasson's Lemma ICRRID} that in some parameters regions the Frasson-Verduyn Lunel's criteria does not apply.
   \end{remark}

In the proof of Theorem~\ref{thm:beyond Frasson's Lemma ICRRID}, we introduced the function
\begin{equation}\label{eq:function Z}
\begin{alignedat}{2}
    {\delta}Z(\tau a,\tau \delta,\tau s_1)&= e^{-\tau\delta}(a+s_1-\delta)(-1+2\tau s_1)\\
    &\quad +(a+s_1)(1-2\tau s_1+\tau\delta).
\end{alignedat}
\end{equation}
When $a=2$ (admissible owing to \eqref{eq:condition on s1 and a}), we plotted the heatmap of $Z(2,\tau \delta,\tau s_1)$ in Figure~\ref{fig:heatmap-Z-W} (image on the \textit{right}) showing the regions $(\tau\delta,\tau s_1)$ where $Y(s_1)<1$ and $Y(s_1)\ge1$.

\subsection{Exponential Estimates}\label{ss:exponential estimates}

In \cite[ Section~6]{kharitonov2005lyapunov}, the author establishes exponential estimates of solutions for time delay systems of neutral-type using quadratic Lyapunov functionals and Lyapunov matrices. Although effective, this method may be computationally involved. Additionally, \cite{hale2013introduction}[Chapter~1, Theorem~7.6, page~32] also offers an exponential estimate for solutions of a neutral differential equation based on the characteristic quasipolynomial's rightmost root. The proof employs the Cauchy theorem of residues and involves complex analysis arguments, including the periodicity and analyticity of certain functions. However, the explanation may benefit from greater detail to enhance the reader's clarity and ease of understanding.

In the following, we provide an alternative proof of exponential estimates for solutions of \eqref{eq:NDE} that integrates the previously outlined spectral analysis of Sections~\ref{ss:3 DRSV} and~\ref{ss:2 DRSV}.

For all $t\ge 0$, $y_t\in C^0([-\tau, 0])$ denotes the history function, defined for all $\theta\in[-\tau,0]$ as
\begin{equation}
    y_t(\theta) = y(t+\theta),\quad\mbox{with}\quad\|y_t\|_\infty:=\sup_{\theta\in[-\tau, 0]}|y_t(\theta)|.
\end{equation}

\begin{theorem}\label{thm:exponential estimates of solutions 3rr}
    Assume that the quasipolynomial $\Delta$ admits three real spectral values $s_3<s_2<s_1$. For every $\varepsilon>0$, there exists a constant $k:=k(\varepsilon,s_1,s_2,s_3)\ge 1$ such that the solution $y(\cdot)$ of \eqref{eq:NDE} with initial condition $y_0\in C^0([-\tau, 0])$ satisfies
    \begin{equation}\label{eq:exponential estimates of solutions 3rr}
        |y(t)|\le k e^{(s_1+\varepsilon)t}\|y_0\|_\infty\quad (t\ge 0).
    \end{equation}
\end{theorem}
\begin{proof}
    Applying the Laplace transform to both sides of equation \eqref{eq:NDE}, one gets
    \begin{equation}
        \hat{y}(s) = \frac{y_0(0)+\alpha y_0(-\tau)}{\Delta(s)}\quad (s\in\C)
    \end{equation}
    showing that $ \hat{y}$ is an analytic function of $s$ for $\Re(s)>s_1$. 
    Let $\varepsilon>0$ and set $c_1:=s_1+\varepsilon$. Then, the function $y$ is given by the Bromwich complex contour integral for all $t\ge 0$
\begin{equation}\label{eq: inverse Laplace transform}
    y(t) = \frac{(y_0(0)+\alpha y_0(-\tau))}{2i\pi}\lim\limits_{T\to\infty}\int_{c_1-iT}^{c_1+iT}\frac{e^{st}}{\Delta(s)}ds.
\end{equation}
Indeed, let $T>0$ and $c_2>c_1$, and consider the integration of the function $e^{ts}/\Delta(s)$ over the closed rectangle $\Gamma$ in the complex plane with vertical boundaries $V_1:=\{c_1+i\omega\mid -T\le\omega\le T\}$ and $V_2:=\{c_2+i\omega\mid -T\le\omega\le T\}$, and horizontal boundaries $H_1 :=\{x+iT\mid c_1\le x\le c_2\}$ and $H_2:=\{x-iT\mid c_1\le x\le c_2\}$. Since $\Delta$ has no zeroes inside the rectangle $\Gamma$, the integral over $\Gamma$ equals zero. It is then sufficient to show the following.
\begin{equation}\label{eq:vanishing of H_j}
    \int_{H_j}\frac{e^{st}}{\Delta(s)}ds\to 0\qquad\mbox{as}\qquad T\to\infty\quad  (j=1,2).
\end{equation}
 For $s=x+iT$ with $c_1\le x\le c_2$ and $T>0$, one has $\Delta(s) = (1+\alpha e^{-\tau s})s+(a+\beta e^{-\tau s})$. Since $1-\alpha e^{-\tau x}>0$ thanks to $s_1<c_1\le x\le c_2$ and Lemma~\ref{lem:properties on alpha and beta GCRRID}, one chooses $T_0>0$ large enough to have
 \begin{equation}\label{eq:T_0 in the case of 3rr}
     \frac{T}{2}(1-\alpha e^{-\tau x})\ge |a|+|\beta|e^{-\tau c_2} \quad(\forall T\ge T_0).
 \end{equation}
 It follows that for all $T\ge T_0$, it holds
 \begin{equation}
     |\Delta(s)|\ge T(1-\alpha e^{-\tau x})-(|a|+|\beta|e^{-\tau x})\ge\frac{T}{2}(1-\alpha e^{-\tau x}).
 \end{equation}
Therefore, one has
 \[
 \left|\int_{H_1}\frac{e^{st}}{\Delta(s)}ds\right|\le\frac{2e^{tc_2}}{T}\int_{c_1}^{c_2}\frac{e^{\tau x}}{(e^{\tau x}-\alpha)}dx=\frac{2}{T}\frac{e^{tc_2}}{\tau}\ln\left(\frac{e^{\tau c_2}-\alpha}{e^{\tau c_1}-\alpha}\right)
 \]
 which tends to $0$ as $T\to\infty$. 
 In the same fashion, one proves that the integral over $H_2$ tends to $0$ when $T\to\infty$. It follows that \eqref{eq: inverse Laplace transform} defines properly the signal $y(t)$ for every $t\ge 0$. One claims that if $s=c_1+iT$, then
\begin{equation}\label{eq:useful large T}
    |\Delta(s)|\ge\frac{|T|}{2}(1-\alpha e^{-\tau s_1}),\quad\forall\;|T|\ge\frac{4\zeta}{1-\alpha e^{-\tau s_1}}=:T_1
\end{equation}
where $\zeta>0$ is defined by \eqref{eq:function zeta} and $1-\alpha e^{-\tau s_1}>0$ owing to Lemma~\ref{lem:properties on alpha and beta GCRRID}.  Indeed, since $a=-s_1-\zeta$ and $\beta=-\alpha s_1+\zeta e^{\tau s_1}$, one has  $\Delta(s) = (1+\alpha e^{-\tau c_1}e^{-i\tau T})(\varepsilon+iT)-\zeta(1-e^{-\tau\varepsilon}e^{-i\tau T})$. Therefore, 
\[
|\Delta(s)|\ge |T|(1-\alpha e^{-\tau c_1})-\zeta(1+e^{-\tau\varepsilon})\ge |T|(1-\alpha e^{-\tau s_1})-2\zeta
\]
thanks to the reverse triangle inequality, which completes the proof of the claim. Setting  for all $t\ge 0$,
\begin{equation}\label{eq:z(t)}
    z(t) := \cL^{-1}\left\{\frac{1}{\Delta(s)}\right\}(t) = \frac{1}{2i\pi}\lim\limits_{T\to\infty}\int_{c_1-iT}^{c_1+iT}\frac{e^{st}}{\Delta(s)}ds 
\end{equation}
one immediately observes that for every $s\in\C$ such that $\Re(s)>s_1$, one has
\[
\frac{1+\alpha e^{-\tau s}}{\Delta(s)} = \zeta\frac{1-e^{-\tau(s-s_1)}}{(s-s_1)\Delta(s)}+\frac{1}{s-s_1}
\]
so that, taking the inverse Laplace transform and the fact that $\cL^{-1}\{1/(s-s_1)\}=e^{s_1t}$, one obtains
\begin{equation}\label{eq:useful with z 3rr}
    z(t)+\alpha z(t-\tau) = \zeta\cL^{-1}\left\{\frac{1-e^{-\tau(s-s_1)}}{(s-s_1)\Delta(s)}\right\}(t)+e^{s_1 t}.
\end{equation}
Thanks to \eqref{eq:useful large T}, one has
\begin{eqnarray}\label{eq:useful intermediate 3rr}
\hspace{-1.5cm} \Lambda_4&:=&\hspace{-0.3cm}\left|\zeta\cL^{-1}\left\{\frac{1-e^{-\tau(s-s_1)}}{(s-s_1)\Delta(s)}\right\}(t)\right|\nonumber\\
    &\le&\hspace{-0.3cm}\frac{\zeta e^{t c_1}}{\pi}\int_{-T_1}^{T_1}\frac{dT}{|\Delta(c_1+iT)|\sqrt{\varepsilon^2+T^2}}\nonumber\\
    &&+\frac{4\zeta e^{t c_1}}{\pi(1-\alpha e^{-\tau s_1})}\int_{T_1}^{\infty}\frac{dT}{T^2}\nonumber\\
    &=&\hspace{-0.3cm}\frac{\zeta e^{t c_1}}{\pi}\hspace{-0.15cm}\int_{-T_1}^{T_1}\frac{dT}{|\Delta(c_1+iT)|\sqrt{\varepsilon^2+T^2}}+\frac{e^{t c_1}}{\pi}\le k_0e^{tc_1}
\end{eqnarray}
where $k_0:=k_0(\varepsilon,s_1,s_2,s_3)>0$ satisfies
\begin{equation}
    \begin{aligned}
        k_0\le\frac{1}{\pi}\left(1+2\zeta T_1\min_{|T|\le T_1}(|\Delta(c_1+iT)|\sqrt{\varepsilon^2+T^2})\right).
    \end{aligned}
\end{equation}
Combining \eqref{eq:useful with z 3rr} and \eqref{eq:useful intermediate 3rr}, one obtains 
\begin{equation}
    |z(t)|-\alpha |z(t-\tau)|\le|z(t)+\alpha z(t-\tau)|\le (1+k_0)e^{(s_1+\varepsilon)t}
\end{equation}
which yields 
\begin{equation}\label{eq:estimate z(t) 3rr}
    |z(t)|\le (1+k_0)e^{(s_1+\varepsilon)t}\sum_{j=0}^{\infty}\alpha^je^{-j(s_1+\varepsilon)\tau} = \frac{(1+k_0)e^{(s_1+\varepsilon)t}}{1-\alpha e^{-\tau(s_1+\varepsilon)}}
\end{equation}
since $0<1-\alpha e^{-\tau(s_1+\varepsilon)}<1$ owing to Lemma~\ref{lem:properties on alpha and beta GCRRID}. As a consequence,
the result follows since
\begin{equation*}\label{eq:estimate y(t) 3rr}
   |y(t)| = |(y_0(0)+\alpha y_0(-\tau))||z(t)|\le \frac{(1+k_0)(1+\alpha)}{1-\alpha e^{-\tau(s_1+\varepsilon)}}e^{(s_1+\varepsilon)t}\|y_0\|_\infty
\end{equation*}
and $k(\varepsilon,s_1,s_2,s_3):={(1+k_0)(1+\alpha)}/{(1-\alpha e^{-\tau(s_1+\varepsilon)})}\ge 1$.
\end{proof}

\begin{remark}
%
It is important to note that \eqref{eq:exponential estimates of solutions 3rr} only make sense when $k\ge 1$. The GCRRID setting derived in Section~\ref{ss:3 DRSV} allows us to explicitly determine $k\ge 1$, a property which is not explicitly stated in \cite{hale2013introduction}[Chapter~1, Theorem~7.6, page~32].
\end{remark}

{Even when two real spectral values coexist in the spectrum of $\Delta$, one has the same exponential estimates as stated in what follows. The proof is a slight adaptation of that of Theorem~\ref{thm:exponential estimates of solutions 3rr}, but we present it for completeness.
\begin{proposition}\label{pro:exponential estimates of solutions 2rr}
    Assume that $\Delta$ has two real spectral values $s_2<s_1$. For every $\varepsilon>0$, there exists a constant $k:=k(\varepsilon,s_1,s_2)\ge 1$ such that the solution $y(\cdot)$ of \eqref{eq:NDE} with initial condition $y_0\in C^0([-\tau, 0])$ satisfies
     \begin{equation}\label{eq:exponential estimates of solutions 2rr}
        |y(t)|\le k e^{(s_1+\varepsilon)t}\|y_0\|_\infty\quad (t\ge 0).
    \end{equation}
\end{proposition}
\begin{proof}
    First, if $s_1$ and $s_2$ are simple, then a third simple real root $s_3$ may coexist in the spectrum of $\Delta$, in which case the result follows by Theorem~\ref{thm:exponential estimates of solutions 3rr}. If such an $s_3$ does not exist, then $s_1$ may be a double real root but remains dominant by Corollary~\ref{cor:condition of the dominance 1rr with multiplicity 2}, and one has $-e^{\tau s_1}\le\alpha\le 0$.\\ \underline{If $\alpha=0$}, then we are in the retarded case, and the exponential estimate \eqref{eq:exponential estimates of solutions 2rr} is proven in \cite{hale2013introduction}[Chapter~1, Theorem~5.2, page~20].\\ 
    \underline{If $-e^{\tau s_1}<\alpha< 0$}, then $s_2<s_1$ are the only simple real spectral values in the spectrum of $\Delta$ and $s_1$ is dominant by Theorem~\ref{thm:dominance of a real root 2 rr}. Then, to prove \eqref{eq:exponential estimates of solutions 2rr}, one argues as in the proof of Theorem~\ref{thm:exponential estimates of solutions 3rr} as follows. Let $\varepsilon>0$ and set $c_1:=s_1+\varepsilon$, and $c_2>c_1$. To prove \eqref{eq:vanishing of H_j}, one can replace \eqref{eq:T_0 in the case of 3rr} by 
    \begin{equation}\label{eq:T_0 in the case of 2rr}
        \frac{T}{2}(1+\alpha e^{-\tau x})\ge |a|+|\beta|e^{-\tau c_2},\qquad \forall T\ge T_0
    \end{equation}
    with $T_0>0$ large enough. In fact, for $s=x+iT$ with $s_1<c_1\le x\le c_2$, one gets from $\Delta(s) = (1+\alpha e^{-\tau s})s+(a+\beta e^{-\tau s})$ that 
    \begin{eqnarray}\label{eq:proof of T_0 in the case of 2rr}
        \hspace{-0.5cm}|\Delta(s)|&\ge&T\sqrt{1+2\alpha e^{-\tau x}\cos(\tau T)+\alpha^2e^{-2\tau x}}-(|a|+|\beta|e^{-\tau x})\nonumber\\
        &\ge&T(1+\alpha e^{-\tau x})-(|a|+|\beta|e^{-\tau x})\ge\frac{T}{2}(1+\alpha e^{-\tau x})\nonumber
    \end{eqnarray}
    owing to \eqref{eq:T_0 in the case of 2rr}, since $\alpha<0$ and $1+\alpha e^{-\tau x}>0$ by \eqref{eq:asymptotic alpha 2RR}. Therefore, one deduces
     \[
 \left|\int_{H_1}\frac{e^{st}}{\Delta(s)}ds\right|\le\frac{2e^{tc_2}}{T}\int_{c_1}^{c_2}\frac{e^{\tau x}}{(e^{\tau x}+\alpha)}dx=\frac{2}{T}\frac{e^{tc_2}}{\tau}\ln\left(\frac{e^{\tau c_2}+\alpha}{e^{\tau c_1}+\alpha}\right)
 \]
 which yields \eqref{eq:vanishing of H_j}. By \eqref{eq:asymptotic alpha 2RR}, one has
 \begin{equation}\label{eq:kappa}
     \kappa:=\frac{\alpha+e^{\tau s_1}}{\tau F_{-\tau, 1}(s_1,s_2)}>0
 \end{equation}
 where $F_{-\tau, 1}(s_1,s_2)$ is defined by \eqref{eq:F 1}. Then, if $s=c_1+iT$, one can replace \eqref{eq:useful large T} by
 \begin{equation}\label{eq:useful large T 2rr}
    |\Delta(s)|\ge\frac{|T|}{2}(1+\alpha e^{-\tau s_1}),\quad\forall\;|T|\ge\frac{4\kappa}{1+\alpha e^{-\tau s_1}}=:T_2.
\end{equation}
In fact, one obtains from \eqref{eq:beta 2 rr} and \eqref{eq:a and alpha} that
 $a=-s_2-\kappa$ and $\beta = -\alpha s_2+\kappa e^{\tau s_2}$ so that $\Delta(s) = (1+\alpha e^{-\tau c_1}e^{-i\tau T})(\delta+\varepsilon+iT)-\kappa(1-e^{-\tau(\delta+\varepsilon)}e^{-i\tau T})$ where $\delta:=s_1-s_2>0$. It follows that
 \[
|\Delta(s)|\ge |T|(1+\alpha e^{-\tau c_1})-\kappa(1+e^{-\tau(\delta+\varepsilon)})\ge |T|(1+\alpha e^{-\tau s_1})-2\kappa
\]
as in the proof of \eqref{eq:T_0 in the case of 2rr}. Hence, \eqref{eq:useful large T 2rr} follows immediately. Now, let $t\ge 0$ and introduce $z(t)$ as in \eqref{eq:z(t)}. One has from  $a=-s_2-\kappa$ and $\beta = -\alpha s_2+\kappa e^{\tau s_2}$ the following
\begin{equation}
    \frac{1+\alpha e^{-\tau s}}{\Delta(s)} = \kappa\frac{1-e^{-\tau(s-s_2)}}{(s-s_2)\Delta(s)}+\frac{1}{s-s_2}
\end{equation}
for every $s\in\C$ such that $\Re(s)>s_1$. Therefore,
\begin{equation}\label{eq:useful with z 2rr}
    z(t)+\alpha z(t-\tau) = \kappa\cL^{-1}\left\{\frac{1-e^{-\tau(s-s_2)}}{(s-s_2)\Delta(s)}\right\}(t)+e^{s_2 t}.
\end{equation}
Arguing as in the proof of \eqref{eq:useful intermediate 3rr}, one obtains from \eqref{eq:useful large T 2rr} that 
\begin{equation}\label{eq:useful intermediate 2rr}
    \left|\kappa\cL^{-1}\left\{\frac{1-e^{-\tau(s-s_2)}}{(s-s_2)\Delta(s)}\right\}(t)\right|\le k_1e^{tc_1}
\end{equation}
where $k_1:=k_1(\varepsilon,s_1,s_2)>0$ satisfies
\begin{equation*}
    \begin{aligned}
        k_1\le\frac{1}{\pi}\left(1+2\kappa T_2\min_{|T|\le T_2}(|\Delta(c_1+iT)|\sqrt{(s_1-s_2+\varepsilon)^2+T^2})\right).
    \end{aligned}
\end{equation*}
Therefore, \eqref{eq:useful with z 2rr} and \eqref{eq:useful intermediate 2rr}, yield
\begin{equation}
    |z(t)|+\alpha|z(t-\tau)|\le|z(t)+\alpha z(t-\tau)|\le (1+k_1)e^{(s_1+\varepsilon)t}
\end{equation}
since $\alpha<0$ by \eqref{eq:asymptotic alpha 2RR}. It follows that
\begin{equation*}\label{eq:estimate z(t) 2rr}
    |z(t)|\le (1+k_1)e^{(s_1+\varepsilon)t}\sum_{j=0}^{\infty}(-\alpha)^je^{-j(s_1+\varepsilon)\tau} = \frac{(1+k_1)e^{(s_1+\varepsilon)t}}{1+\alpha e^{-\tau(s_1+\varepsilon)}}
\end{equation*}
since $0<1+\alpha e^{-\tau(s_1+\varepsilon)}<1$ by \eqref{eq:asymptotic alpha 2RR}. This proves \eqref{eq:exponential estimates of solutions 2rr} since
\begin{equation*}\label{eq:estimate y(t) 2rr}
   |y(t)| = |(y_0(0)+\alpha y_0(-\tau))||z(t)|\le \frac{(1+k_1)(1+|\alpha|)}{1+\alpha e^{-\tau(s_1+\varepsilon)}}e^{(s_1+\varepsilon)t}\|y_0\|_\infty
\end{equation*}
and $k(\varepsilon,s_1,s_2):={(1+k_1)(1+|\alpha|)}/{(1+\alpha e^{-\tau(s_1+\varepsilon)})}\ge 1$.\\
Finally, if $\alpha=-e^{\tau s_1}$, the proof follows \emph{mutatis mutandis} the arguments.
\end{proof}
}

\section{Application to the prescribed local exponential stabilization of a one-layer neural network and beyond}\label{s:application to neural network modelling}

{Building on the spectral theory developed in the previous sections, we apply our theoretical insights to the practical design of an exponentially stable one-layer neural network subject to a delayed PD controller in Section~\ref{ss:stability analysis using PD controller}. Section~\ref{ss:comparison with the P controller} demonstrates that while the proportional (P) controller can achieve the same decay rate as the PD controller, it requires higher gain and only allows for a shorter maximum delay.}


\subsection{Implementing delayed PD controller}\label{ss:stability analysis using PD controller}

In this section, based on the CRRID setting of Section~\ref{ss:3 DRSV}, we practically implement a delayed PD controller to enhance the stability of the Hopfield neural network when $\nu \ge \mu$ and to achieve stabilization when $\nu < \mu$. This is accomplished by determining the gain parameters $k_p$, $k_d$, and the delay $\tau > 0$, with $\mu$ and $\nu$ being known. 


\begin{figure*}[t]
\centering
\includegraphics[width=.45\linewidth]{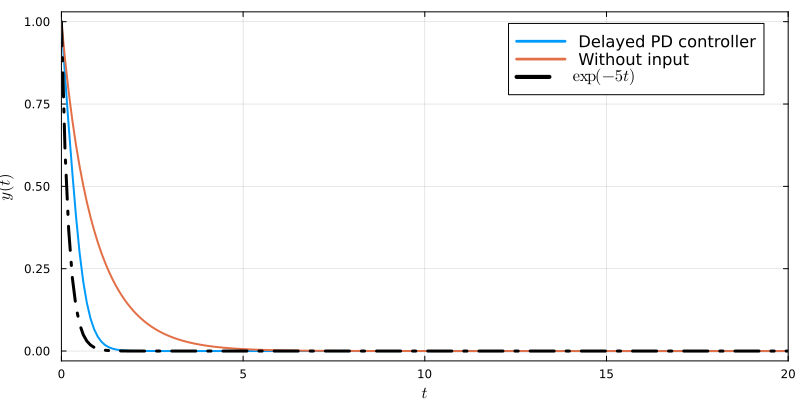}\hspace{0.1em}
\includegraphics[width=.45\linewidth]{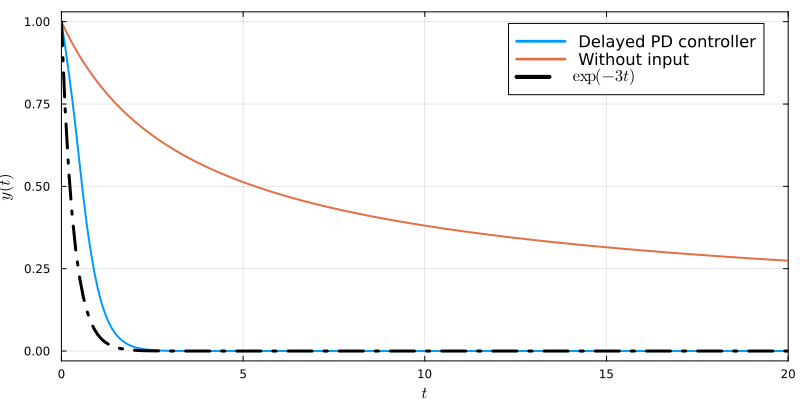}
\vspace{-0.4cm}
\caption{Solutions to \eqref{eq:nu grand que mu} (\textit{left}) and  \eqref{eq:nu egal mu} (\textit{right}) when $I=0$ and when $I$ is the delayed PD controller given by \eqref{eq:feedback control} where $\tau$, $k_d$ and $k_p$ are given by \eqref{eq:coeffs when nu grand que mu-2} (\textit{left}) and \eqref{eq:coeffs when nu egal mu} (\textit{right}) respectively. The initial condition is taken as $y_0 = 1$ for the solution with $I=0$, and $y_0(t)= 1+sin(t)$ for the solution with the delayed PD controller.}
\label{fig:nu-grand_equal-mu}
\end{figure*}

\subsubsection{Improving the decay rate of an exponentially stable one-layer neural network}
Consider the one-layer neural network of the form
\begin{equation}\label{eq:nu grand que mu}
    \dot{y}(t) = - 2y(t) + \tanh(y(t)) + I(t),
\end{equation}
which is a particular case of \eqref{eq:hopfield model nonlinear intrinsec} with $\nu=2$ and $\mu=1$. Under no external input ($I(t) = 0$), all solutions of \eqref{eq:nu grand que mu} are globally exponentially stable, meaning that the trivial equilibrium is globally exponentially stable with a decay rate equal to $\nu - \mu = 1$. A classical control problem is to choose the control $I(t)$ in the feedback form to improve (locally at least) the stability properties of \eqref{eq:nu grand que mu}. By letting $I(t) = -k_p y(t-\tau) - k_d \dot{y}(t-\tau)$ and linearizing around zero, the problem reduces to studying the localization of the spectrum of the quasipolynomial function $Q_0(s) = s+1+e^{-\tau s}(k_d s+k_p)$.

To simplify the control design, we assign three equidistributed real spectral values $s_1=-5$, $s_2=-6$, and $s_3=-7$ to $Q_0$. Then, Lemma~\ref{lem:coeffs}, yields
\begin{equation}\label{eq:coeffs when nu grand que mu-2}
    \tau=\ln\left(3/2\right),\quad k_d = e^{-6\tau},\quad k_p = e^{-6\tau}\left(6+\coth\left(\tau/2\right)\right)
\end{equation}
where $\coth$ is the hyperbolic cotangent function.

As per the CRRID properties provided in Section~\ref{ss:3 DRSV}, the parameters in \eqref{eq:coeffs when nu grand que mu-2} ensure the local exponential stability of solutions to equation \eqref{eq:nu grand que mu} with the PD controller $I(t) = -k_p y(t-\tau) - k_d \dot{y}(t-\tau)$, with a decay rate equal to $-5+\varepsilon$, for a sufficiently small $\varepsilon > 0$, as stated in Theorem~\ref{thm:exponential estimates of solutions 3rr}. Refer to Figure~\ref{fig:nu-grand_equal-mu} (image on the \textit{left}) for a visualization.

\subsubsection{Stabilizing an asymptotically stable one-layer neural network exponentially}
Consider the one-layer neural network described by
\begin{equation}\label{eq:nu egal mu}
    \dot{y}(t) = - y(t) + \tanh(y(t)) + I(t).
\end{equation}
This equation corresponds to a specific case of \eqref{eq:hopfield model nonlinear intrinsec} when $\nu = \mu = 1$. Using a Lyapunov function \cite{hopfield1984neurons}, it can be verified that all solutions of this equation asymptotically converge to zero when $I(t) = 0$. To ensure that the solutions converge to zero with a prescribed exponential decay rate, we control the system using the PD controller \eqref{eq:feedback control}. Linearizing around zero reduces the problem to studying the spectrum of the quasipolynomial $Q_1(s) = s+e^{-\tau s}(k_d s+k_p)$.

Assuming three non-equidistributed real spectral values $s_3 < s_2 < s_1$, we let $s_1 = -3$, $s_2 = -4$, and $s_3 = -6$. Using Lemma~\ref{lem:coeffs}, we find:
\begin{equation}\label{eq:coeffs when nu egal mu}
    \tau = \ln\left(\frac{1+\sqrt{5}}{2}\right),\; k_d = -20 + 9\sqrt{5},\; k_p = -66 + 30\sqrt{5}.
\end{equation}
Figure~\ref{fig:nu-grand_equal-mu} (image on the \textit{right}) illustrates the solutions of \eqref{eq:nu egal mu} with and without the PD controller.

\begin{figure*}[t]
   \centering
\includegraphics[width=.45\linewidth]{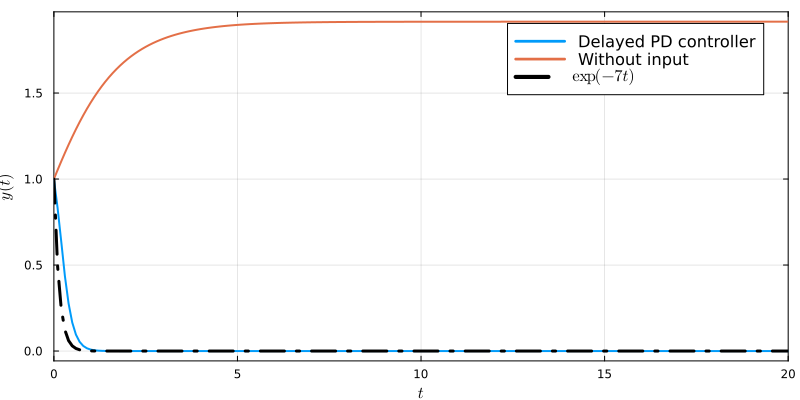}\hspace{0.1em}
    \includegraphics[width=.45\linewidth]{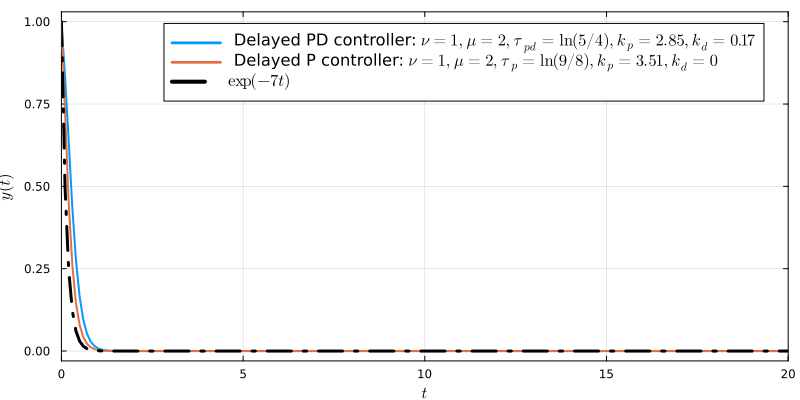}
    \caption{
    On the \textit{left}, solutions to \eqref{eq:nu smaller than mu} when $I=0$ and when $I$ is the delayed PD controller \eqref{eq:feedback control} where $\tau$, $k_d$, and $k_p$ are given by \eqref{eq:coeffs when nu smaller than mu}. 
    { On the \textit{right}, a comparison between the solution to \eqref{eq:nu smaller than mu} with the delayed P and PD controllers. This shows whether they achieve exponential stabilization with the same decay, the total gains of the delayed PD controller are less than that of the delayed P controller. Moreover, the delayed PD controller allows for larger delays than the delayed P controller. See Section~\ref{ss:comparison with the P controller}.} 
    The initial condition is taken as $y_0 = 1$ for the solution with $I=0$, and $y_0(t)= 1+sin(t)$ for the solution with the delayed P and PD controllers. 
    }
    \label{fig:nu-small_control-mu}
\end{figure*}

\subsubsection{Stabilizing an unstable one-layer neural network exponentially}\label{sss:stabilizing an unstable equation}
This section focuses on the scenario where the interaction strength surpasses the neural system's natural decay rate, leading to inherent instability. Specifically, consider the one-layer neural network
\begin{equation}\label{eq:nu smaller than mu}
    \dot{y}(t) = - y(t) + 2\tanh(y(t)) + I(t),
\end{equation}
which is a particular case of \eqref{eq:hopfield model nonlinear intrinsec} with $\nu = 1$ and $\mu = 2$. Linear stability analysis suggests that the trivial equilibrium of \eqref{eq:nu smaller than mu} is unstable when $I = 0$. To stabilize the system, we apply a PD controller \eqref{eq:feedback control}.

By linearizing the equation, the local asymptotic behavior of solutions is equivalent to studying the spectrum of the quasipolynomial function $Q_2(s) = s-1+e^{-\tau s}(k_d s+k_p)$.

Assuming three equidistributed real spectral values $s_1=-7$, $s_2=-8$, and $s_3=-9$, Lemma~\ref{lem:coeffs} yields
\begin{equation}\label{eq:coeffs when nu smaller than mu}
    \tau=\ln\left(5/4\right),\quad k_d = e^{-8\tau},\quad k_p = e^{-8\tau}\left(8+\coth\left(\tau/2\right)\right)
\end{equation}
where $\coth$ is the hyperbolic cotangent function. Figure~\ref{fig:nu-small_control-mu} depicts the solutions of \eqref{eq:nu smaller than mu} with and without the PD controller.

{\subsection{Performance comparison between proportional and proportional-derivative delayed controllers}\label{ss:comparison with the P controller}
In this section, we illustrate why the delayed PD controller, $I(t) = k_p y(t - \tau_{\text{pd}}) - k_d y'(t - \tau_{\text{pd}})$, is more advantageous for the prescribed local exponential stabilization of the one-layer neural network \eqref{eq:hopfield model nonlinear intrinsec} than the delayed proportional (P) controller, $I(t) = -k_p' y(t - \tau_{\text{p}})$.}
{Recall that using $I(t) = -k_p' y(t-\tau_{\text{p}}) $ to locally stabilize $ \dot{y}(t) = -\nu y(t) + \mu \tanh(y(t)) + I(t) $ in the hyperexcitable regime where $ \mu > \nu $, reduces the problem to localizing the spectral abscissa of the quasipolynomial
\begin{equation}
    \Delta_2(s) = a + s + k_p' e^{-\tau_{\text{p}} s}, \qquad s \in \mathbb{C}, \quad (a := \nu - \mu)
\end{equation}
in the left half of the complex plane. By the P\'{o}lya-Szeg\"{o} bound \cite[Problem~206.2, page~144]{polya1972analysis}, $\Delta_2$ is of degree $2$. Assigning two real roots $ s_2 < s_1 $ to $ \Delta_2 $ yields
\begin{equation}\label{eq:a and k_p with 2rr}
    \hspace{-0.3cm}a(\tau_{\text{p}}) = -s_1 - \frac{e^{\tau_{\text{p}} s_2}}{\tau_{\text{p}} F_{-\tau_{\text{p}},1}(s_1, s_2)}, \quad 
    k_p'(\tau_{\text{p}}) = \frac{e^{\tau_{\text{p}} (s_1 + s_2)}}{\tau_{\text{p}} F_{-\tau_{\text{p}},1}(s_1, s_2)}
\end{equation}
where $ F_{-\tau_{\text{p}},1}(s_1, s_2) $ is defined in \eqref{eq:F 1}. From \cite[Theorem~12]{bedouhene2020real}, $ s_1 $ is the dominant root of $ \Delta_2 $, and there exists a unique $ \tau_{\text{p}}^{*} > 0 $ such that $ s_1 < 0 $ if, and only if, $ a(\tau_{\text{p}}^{*}) = 0 $, i.e., 
\begin{equation}\label{eq:tau critical P controller}
    \tau_{\text{p}}^{*} = \frac{1}{(s_1 - s_2)} \ln\left(\frac{s_2}{s_1}\right) > 0.
\end{equation}
For the delayed PD controller, Corollary~\ref{cor:equidistributed three real roots} shows that assigning three equidistributed real roots $ s_3 < s_2 < s_1 $ to the corresponding quasipolynomial $ \Delta_3(s)= s + a + e^{-\tau_{\text{pd}} s}(k_d s + k_p) $ yields
\begin{equation}\label{eq:tau critical PD controller}
    \tau_{\text{pd}}^{*} = \frac{1}{(s_1 - s_2)} \ln\left(\frac{s_3}{s_1}\right) > 0.
\end{equation}
 Therefore, when $s_1<0$, the two controllers ensure the prescribed local exponential stabilization of \eqref{eq:hopfield model nonlinear intrinsec} with the same decay rate as shown in Theorem~\ref{thm:exponential estimates of solutions 3rr} and Proposition~\ref{pro:exponential estimates of solutions 2rr}, see, for instance, image on the \textit{right} in Figure~\ref{fig:nu-small_control-mu}. Comparing \eqref{eq:tau critical P controller} and \eqref{eq:tau critical PD controller}, it follows that the delayed PD controller allows for larger delays than the delayed P controller, as $ \tau_{\text{pd}}^{*} > \tau_{\text{p}}^{*} $. This inequality holds even in non-equidistributed scenarios, though determining $ \tau_{\text{pd}}^{*} $ requires numerical computations, as it cannot be easily derived analytically.}

{Finally, when the inherent dynamics are unstable ($ \mu > \nu $), the delayed P controller requires a significantly larger gain $ k_p' $ than the combined gains $ k_p $ and $ k_d $ of the delayed PD controller. For simplicity, assume $ \nu = 1 $ and $ \mu = 2 $. Assigning three equidistributed real roots $ s_3 = -9 $, $ s_2 = -8 $, and $ s_1 = -7 $ for the PD controller gives
\begin{equation}\label{eq:kd}
    \tau_{\text{pd}} = \ln\left(5/4\right) \approx 0.22, \quad k_d = e^{-8\tau_{\text{pd}}} \approx 0.17,
\end{equation}
\begin{equation}\label{eq:kp}
    k_p = e^{-8\tau_{\text{pd}}} \left(8 + \coth\left(\tau_{\text{pd}} / 2\right)\right) \approx 2.85.
\end{equation}
In contrast, assigning two real roots $ s_2 = -8 $ and $ s_1 = -7 $ for the P controller yields
\begin{equation}
    \tau_{\text{p}} = \ln(9/8) \approx 0.12, \quad k_p' \approx 3.51.
\end{equation}
Clearly, $ \tau_{\text{pd}} > \tau_{\text{p}} $, $ k_p' > \max(k_p, k_d) $, $ k_p' > k_p + k_d $, and $ k_p' > \sqrt{k_p^2 + k_d^2} $. Moreover, decreasing $ s_2$ in the P controller design decreases $ \tau_{\text{p}}$ and increases the gain $ k_p'$, making the design even less favorable compared to the PD controller. For instance, with $ s_2 = -9$ and $ s_1 = -7$, we find
\begin{equation}
    \tau_{\text{p}} = \ln(\sqrt{5}/2) \approx 0.11, \qquad k_p' \approx 3.66.
\end{equation}
Conversely, increasing $ s_2$ increases the delay $ \tau_{\text{p}}$ and decreases the gain $ k_p'$. Furthermore, even in the limiting case where $ s_2 = s_1 = -7$, \cite[Theorem~3.1]{mazanti2021multiplicity} ensures that $ -7$ remains the dominant root of $ \Delta_2(s) = s - 1 + k_p' e^{-\tau_{\text{p}} s}$, yielding
\begin{equation}
    \tau_{\text{p}} = \frac{1}{8} = 0.125, \qquad k_p' = 3.33.
\end{equation}
This scenario still satisfies $ \tau_{\text{pd}} > \tau_{\text{p}}$, $ k_p' > \max(k_p, k_d)$, $ k_p' > k_p + k_d$, and $ k_p' > \sqrt{k_p^2 + k_d^2}$, where $ \tau_{\text{pd}}$, $ k_p$, and $ k_d$ are given by \eqref{eq:kp} and \eqref{eq:kd}. However, even this ``better'' scenario ($ s_2 = s_1 = -7$) is not appropriate in practice, as noted in \cite{michiels2017explicit}, because non-semi-simple spectral values are highly sensitive to minor perturbations due to their splitting mechanism.} 

{Let $k$ be the gain vector ($k\in\R$ for the delayed P controller and $k\in\R^2$ for the delayed PD controller). If one wants to stabilize with a constraint $|k|\le\gamma$ for some $\gamma>0$ where $|k|$ is the Euclidean norm of $k$, then the delayed PD controller is more likely to satisfy this constraint. Note that constrained stabilization of scalar linear delay systems with two delays has been considered in \cite[Section~5]{fueyo2023pole}.}

\section{Discussion}

This paper demonstrates the use of a delayed Proportional-Derivative (PD) controller in the continuous-time modeling of a one-layer Hopfield-type neural network to achieve local exponential stabilization. Through rigorous spectral analysis, we have developed a methodological framework that enhances the stability of the model, focusing on achieving a prescribed exponential decay rate. We have extended the spectral theory based on the CRRID property for linear functional differential equations of neutral type to neural dynamics, providing a powerful analytical tool for examining the stability of neural networks through their spectral properties. This has allowed us to determine conditions under which the network attains stability, even in systems where conventional methods predict instability.

{Our analysis reveals that while a delayed proportional controller can achieve the same decay rate as the delayed PD controller, it requires significantly higher gains and permits a shorter maximum allowable delay. }

Future research should explore the implementation of the delayed PD control strategy and the CRRID setting in more complex, multi-layer neural network architectures, which could better represent the intricate structures of biological neural systems. The MID setting developed in \cite{benarab2023multiplicity, boussaada2022generic, mazanti2021multiplicity} may also address this issue. As noted in \cite{michiels2017explicit}, non-semi-simple spectral values are sensitive to minor perturbations due to their splitting mechanism, which presents a challenge in practical implementations.

This study provides valuable insights into the local stability of the zero equilibrium in the nonlinear model. It is important to note that these findings primarily concern local dynamics. Future research should combine the spectral methods used for the linear equation with time-domain approaches based on Lyapunov functionals and linear matrix inequalities. Such a combined approach would facilitate a more comprehensive investigation into the global exponential stability of the nonlinear equation with a prescribed decay rate, broadening the applicability and reliability of the results presented here.


\section*{Acknowledgments}
The authors wish to thank their colleagues  {\sc{Fazia Bedouhene}} (LMPA, The Mouloud Mammeri University of Tizi Ouzou, Algeria), {\sc{Antoine Chaillet}}, {\sc{Guilherme Mazanti}} and {\sc{Silviu-Iulian Niculescu}} (L2S, University Paris-Saclay, France), and {\sc{Timoth\'ee Schmoderer}} (Prisme, University of Orl\'eans, France)  for valuable discussions on the CRRID property and references suggestions on time-domain approaches based on Lyapunov functionals.  We also thank the anonymous referees for carefully reading the paper and for their valuable comments that helped us enhance the readability.


\begin{thebibliography}{00}


\bibitem{amrane2018qualitative}
  Amrane, Souad and Bedouhene, Fazia and Boussaada, Islam and Niculescu, Silviu-Iulian,
  \textit{On qualitative properties of low-degree quasipolynomials: further remarks on the spectral abscissa and rightmost-roots assignment},
  Bulletin math{\'e}matique de la Soci{\'e}t{\'e} des Sciences Math{\'e}matiques de Roumanie,
  61(4):361–381,
  2018.

  \bibitem{beuter1993feedback}
  Beuter, Anne and B{\'e}lair, Jacques and Labrie, Christiane,
  \textit{Feedback and delays in neurological diseases: A modeling study using dynamical systems},
  Bulletin of mathematical biology,
  55:525–541,
  1993.

  \bibitem{bedouhene2020real}
  Bedouhene, Fazia and Boussaada, Islam and Niculescu, Silviu-Iulian,
  \textit{Real spectral values coexistence and their effect on the stability of time-delay systems: Vandermonde matrices and exponential decay},
  Comptes Rendus. Math{\'e}matique,
  58(9-10):1011-1032,
  2020.

\bibitem{benarab2023multiplicity}
  Benarab, Amina and Boussaada, Islam and Niculescu, Silviu-Iulian and Trabelsi, Karim L,
  \textit{Multiplicity-induced-dominance for delay systems: Comprehensive examples in the scalar neutral case},
  European Journal of Control,
  74:100835,
  2023.

\bibitem{boussaada2022generic}
  Boussaada, Islam and Mazanti, Guilherme and Niculescu, Silviu-Iulian,
  \textit{The generic multiplicity-induced-dominance property from retarded to neutral delay-differential equations: When delay-systems characteristics meet the zeros of Kummer functions},
  Comptes Rendus. Math{\'e}matique,
  360(G4):349–369,
  2022.

\bibitem{boussaada2020multiplicity}
  Boussaada, Islam and Niculescu, Silviu-Iulian and El-Ati, Ali and P{\'e}rez-Ramos, Redamy and Trabelsi, Karim,
  \textit{Multiplicity-induced-dominance in parametric second-order delay differential equations: Analysis and application in control design},
  ESAIM: Control, Optimisation and Calculus of Variations,
  26:57,
  2020.


  \bibitem{Coron2015feedback}
  Coron, Jean Michel and Tamasoiu, Simona Oana,
  \textit{Feedback stabilization for a scalar conservation law with {PID} boundary control},
  Chinese Annals of Mathematics, Series B,
  36(5):763–776,
  2015.

  \bibitem{depannemaecker2021modeling}
  Depannemaecker, Damien and Destexhe, Alain and Jirsa, Viktor and Bernard, Christophe,
  \textit{Modeling seizures: From single neurons to networks},
  Seizure,
  90:4–8,
  2021.


\bibitem{fueyo2023pole}
Fueyo, Sébastien, Mazanti, Guilherme, Boussaada, Islam, Chitour, Yacine, and Niculescu, Silviu-Iulian,
\textit{On the pole placement of scalar linear delay systems with two delays},
IMA Journal of Mathematical Control and Information,
40(1):81–105,
2023.

\bibitem{fridman2001new}
  Fridman, Emilia,
  \textit{New Lyapunov--Krasovskii functionals for stability of linear retarded and neutral type systems},
  Systems \& control letters,
  43(4):309–319,
  2001.

  \bibitem{fridman2014introduction}
  Fridman, Emilia,
  \textit{Introduction to time-delay systems: Analysis and control},
  Springer,
  2014.

\bibitem{frasson2003large}
  Frasson, Miguel VS and Verduyn Lunel, Sjoerd M,
  \textit{Large time behaviour of linear functional differential equations},
 Integral Equations and Operator Theory,
   47:91–121,
  2003.

  \bibitem{gerstner2014}
  Gerstner, Wulfram and Kistler, Werner M. and Naud, Richard and Paninski, Liam,
  \textit{Neuronal Dynamics: From Single Neurons to Networks and Models of Cognition},
 Cambridge University Press,
  2014.

  \bibitem{gopalsamy2013stability}
  Gopalsamy, Kondalsamy,
  \textit{Stability and oscillations in delay differential equations of population dynamics},
 Springer Science \& Business Media,
   volume 74,
  2013.

\bibitem{hale2013introduction}
 Hale, Jack K and Lunel, Sjoerd M Verduyn,
  \textit{Introduction to functional differential equations},
 Springer Science \& Business Media,
   volume 99,
  2013.

\bibitem{hopfield1984neurons}
  Hopfield, John J,
  \textit{Neurons with graded response have collective computational properties like those of two-state neurons},
  Proceedings of the national academy of sciences,
  81(10):3088–3092,
  1984.

  \bibitem{kharitonov2005lyapunov}
  Kharitonov, VL,
  \textit{Lyapunov functionals and Lyapunov matrices for neutral type time delay systems: a single delay case},
  International Journal of Control,
  78(11):783–800,
  2005.

\bibitem{kuang1993delay}
  Kuang, Yang,
  \textit{Delay differential equations: with applications in population dynamics},
  Academic press,
  1993.

\bibitem{lien2000stability}
  Lien, Chang-Hua and Yu, Ker-Wei and Hsieh, Jer-Guang,
  \textit{Stability conditions for a class of neutral systems with multiple time delays},
  Journal of Mathematical Analysis and Applications,
  245(1):20–27,
  2000.

\bibitem{michiels2017explicit}
  Michiels, Wim and Boussaada, Islam and Niculescu, Silviu-Iulian,
  \textit{An explicit formula for the splitting of multiple eigenvalues for nonlinear eigenvalue problems and connections with the linearization for the delay eigenvalue problem},
  SIAM Journal on Matrix Analysis and Applications,
  38(2):599–620,
  2017.

\bibitem{mazanti2021multiplicity}
  Mazanti, Guilherme and Boussaada, Islam and Niculescu, Silviu-Iulian,
  \textit{Multiplicity-induced-dominance for delay-differential equations of retarded type},
  Journal of Differential Equations,
  286:84–118,
  2021.


\bibitem{miranda2010firing}
Miranda-Dominguez, O., Gonia, J., and Netoff, T. I.,
\textit{Firing rate control of a neuron using a linear proportional-integral controller},
Journal of Neural Engineering,
7(6):066004,
2010.

\bibitem{petkoski2019transmission}
  Petkoski, Spase and Jirsa, Viktor K,
  \textit{Transmission time delays organize the brain network synchronization},
  Philosophical Transactions of the Royal Society A,
  377(2153):20180132,
  2019.

  \bibitem{polya2012problems}
  P{\'o}lya, George and Szeg{\"o}, Gabor,
  \textit{Problems and theorems in analysis II: theory of functions. Zeros. Polynomials. Determinants. Number theory. Geometry},
  Springer Science \& Business Media,
  2012.

  \bibitem{polya1972analysis}
  P{\'o}lya, George and Szeg{\"o}, Gabor,
  \textit{Problems and theorems in analysis I: Series. Integral. Calculus . Theory of Functions},
  Springer-Verlag Berlin Heidelberg GmbH,
  1972.
  
  \bibitem{ruan2006delay}
  Ruan, S.,
  \textit{Delay differential equations in single species dynami},
  Springer Netherlands,
  477--517,
  2006.

\bibitem{schmoderer2023boundary}
  Schmoderer, Timoth{\'e}e and Boussaada, Islam and Niculescu, Silviu-Iulian,
  \textit{On Boundary Control of the Transport Equation. Assigning Real Spectra \& Exponential Decay},
  IEEE Control Systems Letters,
  IEEE,
  2023.

  \bibitem{schmoderer2023insights}
  Schmoderer, Timoth{\'e}e and Boussaada, Islam and Niculescu, Silviu-Iulian,
  \textit{Insights on equidistributed real spectral values in second-order delay systems: perspectives in partial pole placement},
  Systems \& Control Letters,
  185:105728,
  2024.

   \bibitem{sreenivasan2019and}
 Sreenivasan, Kartik K and D’Esposito, Mark,
  \textit{The what, where and how of delay activity},
  Nature reviews neuroscience,
  20(8):466–481,
  2019.


  \bibitem{stepan2009delay}
 Stepan, Gabor,
  \textit{Delay effects in the human sensory system during balancing},
  Philosophical Transactions of the Royal Society A: Mathematical, Physical and Engineering Sciences,
 367(1891):1195–1212,
  2009.

   \bibitem{wei1999stability}
 Wei, Junjie and Ruan, Shigui,
  \textit{Stability and bifurcation in a neural network model with two delays},
  Physica D: Nonlinear Phenomena,
 130(3-4):255–272,
  1999.

   \bibitem{yaffe2015physiology}
 Yaffe, Robert B and Borger, Philip and Megevand, Pierre and Groppe, David M and Kramer, Mark A and Chu, Catherine J and Santaniello, Sabato and Meisel, Christian and Mehta, Ashesh D and Sarma, Sridevi V,
  \textit{Physiology of functional and effective networks in epilepsy},
  Clinical Neurophysiology,
 126(2):227–236,
  2015.

  

\end{thebibliography}
\end{document}